\documentclass[11pt]{article}
\usepackage{url}
\usepackage{amsfonts, mathrsfs}
\usepackage{todonotes}
\usepackage{amssymb}
\usepackage{amsmath}
\usepackage{amsthm}
\usepackage[lofdepth,lotdepth]{subfig}
\usepackage{setspace,amssymb,verbatim, enumitem}
\usepackage[margin=1in]{geometry}

\usepackage{lineno}

\usepackage{prettyref}
\newrefformat{cond}{condition (\ref{#1})}
\newrefformat{conj}{Conjecture \ref{#1}}
\usepackage{graphicx}
\usepackage{rotating}
\usepackage{psfrag}
\usepackage{afterpage}
\usepackage{multirow}
\usepackage[lofdepth,lotdepth]{subfig}
\usepackage{tikz}

\theoremstyle{plain}
\newtheorem{theorem}{Theorem}[section]
\newtheorem{proposition}[theorem]{Proposition}
\newtheorem*{theorem*}{Theorem}

\newtheorem*{corollary*}{Corollary}
\newtheorem{corollary}[theorem]{Corollary}

\newtheorem*{lemma*}{Lemma}
\newtheorem{lemma}[theorem]{Lemma}

\theoremstyle{remark}

\newtheorem{remark}[theorem]{Remark}

\newtheorem*{case*}{Case}

\theoremstyle{definition}
\newtheorem{notation}[theorem]{Notation}
\newtheorem{definition}[theorem]{Definition}

\usepackage{prettyref}
\newrefformat{sec}{Section \ref{#1}}
\newrefformat{chap}{Chapter \ref{#1}}
\newrefformat{prop}{Proposition \ref{#1}}
\newrefformat{prob}{Problem \ref{#1}}
\newrefformat{rem}{Remark \ref{#1}}
\newrefformat{cor}{Corollary \ref{#1}}
\newrefformat{cond}{Condition \ref{#1}}
\newrefformat{def}{Definition \ref{#1}}
\newrefformat{ex}{Example \ref{#1}}
\newrefformat{lem}{Lemma \ref{#1}}
\newrefformat{eq}{Equation \eqref{#1}}
\newrefformat{con}{condition (\ref{#1})}
\newrefformat{conj}{Conjecture \ref{#1}}
\newrefformat{fig}{Figure \ref{#1}}
\newrefformat{tab}{Table \ref{#1}}
\newrefformat{app}{Appendix \ref{#1}}

\newcommand{\Z}{\mathbb{Z}}

\newcommand{\aut}{\mathrm{Aut}}
\newcommand{\out}{\mathrm{Out}}

\newcommand{\C}{\mathbb{C}}
\newcommand{\R}{\mathbb{R}}
\newcommand{\F}{\mathbb{F}}
\newcommand{\E}{\mathcal{E}}

\newcommand{\Irr}{\mathrm{Irr}}
\newcommand{\ind}{\mathrm{Ind}}
\newcommand{\irr}{\mathrm{Irr}}
\newcommand{\cusp}{\mathrm{Irr}_{\mathrm{cusp}}}

\newcommand{\Bl}{\mathrm{Bl}}
\newcommand{\res}{\mathrm{Res}}

\newcommand{\wh}[1]{\widehat{#1}}

\newcommand{\Aut}{\mathrm{Aut}}

\newcommand{\la}{\lambda}
\newcommand{\si}{\sigma}
\newcommand{\al}{\alpha}
\newcommand{\ov}{\overline}
\newcommand{\bt}[1]{\mathrm{#1}}
\newcommand{\bb}[1]{\mathbb{#1}}

\newcommand{\wt}[1]{\widetilde{#1}}
\newcommand{\bG}[1]{\mathbf{#1}}
\newcommand{\bg}[1]{\mathbf{#1}}
\newcommand{\xx}{\bg{x}}

\newcommand{\wrt}{{with respect to\ }}
\newcommand{\Sym}{{\mathcal{S}}}
\newcommand{\NNN}{\operatorname{N}}
\newcommand{\Cent}{\operatorname{C}}
\newcommand{\Zent}{\operatorname{Z}}
\newcommand{\Inn}{\operatorname{Inn}}
\newcommand{\uad}{{\underline{a_d}}}
\usepackage{cleveref,enumitem}

\newlist{asslist}{enumerate}{1} 
\setlist[asslist]{label=(\roman*), 
}

\crefalias{asslistenumi}{Assumption} 

\newlist{thmlist}{enumerate}{1} 
\setlist[thmlist]{label=(\alph*), ref=(\alph*)
}

\newcommand{\calG}{\mathcal G}\newcommand{\cD}{\mathcal D}
\newcommand{\calM}{\mathcal M}
\newcommand{\calB}{\mathcal B}
\newcommand{\onjd}{\overline n_j^{(d)}}
\def\spann<#1>{\left\langle#1\right\rangle}

\makeatletter
\def\Set#1{\Set@h#1@}
\def\Lset#1{\Lset@h#1@}
\def\Set@h#1|#2@{\left\{\left.#1\vphantom{#2}\hskip.1em\,\right|\,\relax #2\right\}}
\def\Lset@h#1@{\left\{#1\right\}}
\def\CALC#1{\CALC@h#1@}
\def\CALC@h#1|#2@{\calC^{#1}(#2)}
\def\CALCrad#1{\CALCrad@h#1@}
\def\CALCrad@h#1|#2@{\calC_\radic^{#1}(#2)}

\def\CALCNC#1{\CALCNC@h#1@}
\def\CALCNC@h#1|#2@{\calC_{\radic,nc}^{#1}(#2)}

\def\restr#1|#2{\left.#1\right\rceil_{#2}}
\def\spann<#1>{\left\langle#1\right\rangle}

\def\Spann<#1>{\Spann@h#1@}
\def\Spann@h#1|#2@{\left\langle\left.#1\vphantom{#2}\hskip.1em\right.\mid\relax #2 \right\rangle}
\def\Set#1{\Set@h#1@}
\def\Set@h#1|#2@{\left\{\left.#1\vphantom{#2}\hskip.1em\,\right.
	\mid\relax #2\right\}}
\def\set#1{\set@h#1@}
\def\set@h#1@{\left\{#1\right\}}
\def\spann<#1>{\left\langle#1\right\rangle}
\makeatother

\title{On the Inductive Alperin-McKay Conditions in the Maximally Split Case}
\author{
Marc Cabanes \\ \small\textit{Institut de Math{\'e}matiques de Jussieu - PRG} \\ \small\textit{CNRS} \\ \small \textit{ 8 Place Aur\'elie Nemours,
	75205 Paris Cedex 13, France} \\ \small \textit{cabanes@math.jussieu.fr}
\and
A. A. Schaeffer Fry\\ \small\textit{Department of Mathematical and Computer Sciences}\\\small\textit{Metropolitan State University of Denver}\\ \small\textit{Denver, CO 80217, USA}\\ \small\textit{aschaef6@msudenver.edu}
\and
Britta Sp{\"a}th\\ \small\textit{Universit{\"a}t Wuppertal}\\\small\textit{School of Mathematics and Natural Sciences}\\\small\textit{ Gau\ss str. 20, 42119 Wuppertal, Germany}\\ \small\textit{bspaeth@uni-wuppertal.de}
}

\date{\today}
\begin{document}
\maketitle

\begin{abstract}

The Alperin-McKay conjecture relates height zero characters of an $\ell$-block with the ones of its Brauer correspondent. This conjecture has been reduced to the so-called inductive Alperin-McKay conditions about quasi-simple groups by the third author. Those conditions are still open for groups of Lie type. The present paper describes characters of height zero in $\ell$-blocks of groups of Lie type over a field with $q$ elements when $\ell$ divides $q-1$. We also give information about $\ell$-blocks and Brauer correspondents.  As an application we show that quasi-simple groups of type $\sf C$ over $\F_q$ satisfy the inductive Alperin-McKay conditions for primes $\ell\geq 5$ and dividing $q-1$. Some methods to that end are adapted from \cite{MalleSpathMcKay2}.
\vspace{0.25cm}

\noindent \textit{Mathematics Classification Number:} 20C15, 20C33


\end{abstract}

\section{Introduction} 

The well-known McKay conjecture from 1972 posits that for a finite group $G$ and a prime $\ell$ dividing $|G|$, there should be a bijection between the irreducible characters of degree prime to $\ell$ of $G$ and those of $\NNN_G(P)$ for a Sylow $\ell$-subgroup $P$ of $G$.  The blockwise version of the McKay conjecture, known as the Alperin-McKay Conjecture, states that the number of height-zero characters of an $\ell$-block $B$ of $G$ with defect group $D$ should be the same as the number of height-zero characters of the Brauer correspondent of $B$ in $\NNN_G(D)$.

Reduction theorems for the McKay and Alperin-McKay conjectures are proven in \cite{IsaacsMalleNavarroMcKayreduction} and \cite{spathAMreduction}, respectively.  In particular, in each case it is shown that to prove the conjecture, it suffices to prove certain ``inductive" conditions for all finite non-abelian simple groups.  From \cite{spathAMreduction} and \cite{Denoncin_AM}, we know that the alternating groups satisfy the inductive Alperin-McKay conditions and that the simple groups of Lie type satisfy the inductive Alperin-McKay conditions when $\ell$ is the defining characteristic. The situation that a simple group has abelian Sylow subgroups was considered in \cite{Malle14}, and certain low-rank cases have been settled in \cite{Malle14, SchaefferFryGood}.  Further, \cite{cabanesspath15, broughspath} consider the case of groups of Lie type ${\sf A}$.

In the present paper, we describe the height-zero characters in blocks of finite groups of Lie type and of an appropriate subgroup containing the normalizer of a defect group for certain good primes.  We prove sufficient conditions for a group of Lie type $G$ in this situation to satisfy the inductive Alperin-McKay conditions and, as an application, further prove that if $G$ is of type $\sf C$, then these conditions hold.  That is, we prove:

\begin{theorem}\label{thm:iAMtypeC}
The simple groups {\rm PSp}$_{2n}(q)$ with $q$ odd and $n\geq 2$ satisfy the inductive Alperin-McKay conditions from \cite[7.3]{spathAMreduction} for primes $\ell\geq 5$ dividing $(q-1)$.
\end{theorem}

Section 2 deals with height zero characters in $\ell$-blocks of groups of Lie type over $\F_q$ when $\ell$ divides $q-1$. Then Section 3 gives a streamlined version of the inductive Alperin-McKay conditions in that case, see Proposition~\ref{prop:ourconds}. Section 4 uses some methods from \cite{MalleSpathMcKay2} and the description of normalizers of split Levi subgroups to check those conditions in the case of finite symplectic groups.

\subsection{Notation for characters and blocks}

Given finite groups $H\leq G$, we write $\irr(G)$ for the set of irreducible (complex) characters of $G$, $\irr(G\mid\varphi)$ for the set of irreducible constituents of $\varphi^G:=\ind_H^G(\varphi)$ when $\varphi\in\irr(H)$, and $\irr(H\mid \chi)$ for the set of irreducible constituents of $\chi|_H:=\res^G_H(\chi)$ for $\chi\in\irr(G)$.  More generally, for any  subset $X\subseteq \irr(H)$ we write $\irr(G\mid X):=\cup_{\varphi\in X}\irr(G\mid\varphi)$. 

Let $\ell$ be a prime number. Given a defect $\ell$-subgroup $D$ of $G$, we write $\irr(G\mid D)$ and $\irr_0(G\mid D)$ for the set of irreducible characters lying in an $\ell$-block with defect group $D$ and the set of those characters with height $0$ within their block, respectively.  We denote by $\Bl(G)$ the set of $\ell$-blocks of $G$ and whenever $\chi\in\irr(G)$, we write $b_G(\chi)$ for the block of $G$ containing $\chi$.  We will write $G_\chi$, respectively $G_B$, for the stabilizer in a group $G$ of a character $\chi$, respectively block $B$, of some normal subgroup. For $b$ a block of some subgroup of $G$, we denote by $b^G$ the corresponding Brauer induced block of $G$ when defined (see \cite[p. 87]{NavarroBlocks}).

For any integer $n$, we write $n_\ell$ for the largest power of $\ell$ dividing $n$.  Further, for an abelian group $H$, we write $H_\ell$ for the Sylow $\ell$-subgroup of $H$.

Finally, for $H\lhd G$, the following definition will be useful.

\begin{definition}\label{ExtMap}
Let $H\lhd G$ and let $\Bbb I$ be a subset of $\Irr(H)$.	An {\it extension map with respect to $H\lhd G$ for} $\Bbb I$ is any map $$\Lambda\colon{\Bbb I}\to\coprod_{G'\ :\ H\lhd G'\leq G}\Irr(G')$$ associating to each $\varphi\in\Bbb I$ an extension $\Lambda (\varphi)$ of $\varphi$ in $\Irr(G_\varphi)$.
\end{definition}

\section{Constructing the Bijection}\label{sec:map}

The inductive Alperin-McKay conditions from \cite{spathAMreduction} require a bijection between height-zero characters having certain properties. In the present section, we introduce the finite groups of Lie type and we describe a bijection of characters for certain primes $\ell$ (see Corollary~\ref{cor:map}).

\subsection*{The Framework and More Notation}

We refer to \cite{dignemichel} for characters of finite groups of Lie type.
Throughout this section, we let $G=\bg{G}^F$ be a group of Lie type defined over $\F_q$, where $q$ is a power of a prime $p$, $\bg{G}$ is a connected reductive algebraic group, and $F$ is a Frobenius endomorphism on $\bg{G}$.  Further let $(\bg{G}^\ast, F^\ast)$ be dual to $(\bg{G}, F)$ and let $G^\ast:={\bg{G}^\ast}^{F^\ast}$.

We write $\mathcal{E}(G,s)$ for the rational Lusztig series corresponding to the conjugacy class of the semisimple element $s\in G^\ast$ (see \cite[14.41]{dignemichel}).  Recall that $\mathcal{E}(G,s)$ depends only on the conjugacy class of $s$ in $G^*$ and that $\irr(G)$ is the disjoint union of the various sets $\mathcal{E}(G,s)$.  Let $\ell$ be a prime not dividing $q$.   If $s$ is a semisimple $\ell'$-element, then we write $\mathcal{E}_\ell(G,s)$ for the union of series of the form $\mathcal{E}(G, st)$, where the union ranges over $\ell$-elements $t\in \Cent_{G^\ast}(s)$.  By Brou\'e-Michel's theorem \cite[9.12(i)]{cabanesenguehard}, $\mathcal{E}_\ell(G,s)$ is also a union of $\ell$-blocks.  We will write $\mathcal{E}(G,\ell')$ for the union $\cup_s \mathcal{E}(G,s)$ where $s$ ranges over semisimple $\ell'$-elements of $G^*$.  

Now, we let $\bg{L}$ be a fixed split Levi subgroup of $\bg{G}$ (i.e. $\bg{L}$ is $F$-stable and the Levi supplement of some $F$-stable parabolic subgroup) and $L:=\bg{L}^F$ the corresponding Levi subgroup of $G$.  We fix a character $\lambda\in\cusp(L)$, where $\cusp(L)$ is the set of irreducible cuspidal characters of $L$, so that $(L,\la)$ is a cuspidal pair (see \cite[Ch. 6]{dignemichel}).  Further, assume that $\la\in\mathcal{E}(L, \ell')$.

Assume $\ell$ divides $q-1$ and let $b:=b_L(\lambda)$ denote the $\ell$-block of $L$ containing $\lambda$, which by the main results of \cite{CabanesEnguehard99, KessarMalle15} often (and in particular in the situations considered here) parametrizes a block $B:=b_G(L,\lambda)$ of $G$.  A defining property (see \cite[4.1(a)]{CabanesEnguehard99}) is that $B$ contains the constituents of $R_L^G(\lambda)$, where $R_L^G$ denotes here Harish-Chandra induction (see \cite[4.6(iii), 6.1]{dignemichel}).   Further, we have $B=b^G$ (see \prettyref{prop:HCblock} below).

Let $N:=\NNN_{\bg{G}}(\bg{L})^F$ be the fixed points under $F$ of the normalizer of $\bg{L}$ in $\bg{G}$ and let $\wt{b}$ be a block of $N$ lying above $b$.  Further, for a block $c$, let $D(c)$ denote a fixed defect group for $c$.

 For a cuspidal pair $(L,\psi)$, the irreducible constituents of $R_L^G(\psi)$ are in bijection with the irreducible characters of $W(\psi):=N_\psi/L$, and we will write the constituent corresponding to $\eta\in\irr(W(\psi))$ as $R_L^G(\psi)_\eta\in\irr(G)$, as in \cite[4.D]{MalleSpathMcKay2}.

\subsection{First Steps: The Global Side}

\begin{proposition}\label{prop:HCblock}
Let $G=\bg{G}^F$ be a finite group of type as in the previous section, with $F$ defining $\bf G$ over a field with $q$ elements. 
Let $\ell$ be a prime good for $\bf G$ and dividing $q-1$. Let $\psi\in \irr(L)$ (not necessarily cuspidal) for $L:=\bg{L}^F$ with $\bg L$ a split Levi subgroup of $\bg G$. Assume $\ell\nmid \, [ \Zent(\bG{G})^F:\Zent^\circ(\bg{G})^F]$. Then $L=\Cent_G(\Zent(L)_\ell)$ and

\[R_L^G(\psi)\in \Z\Irr(b^G)\]
where $b$ is the $\ell$-block of $\psi$ in $L$.
\end{proposition}
\begin{proof} The first equality comes from \cite[3.2]{CabanesEnguehard99}. We now check the second statement.
	
First, note that it suffices to show that all constituents of $R_L^G(\psi)$ lie in the same block $B$, using \cite[6.4]{NavarroBlocks} for example, since $R_L^G(\psi)$ is the induction of the inflation of $\psi$ and hence we must have $b^G=B$. 

If $\psi\in\mathcal{E}(L,\ell')$, then the statement follows from \cite[2.5]{CabanesEnguehard99}.  So assume $\psi\not\in\mathcal{E}(L,\ell')$ and let $b$ be a block of $L$ with $\irr(B)\subseteq \mathcal{E}_\ell(L, s)$ for some semisimple $\ell'$--element $s$ of $L^\ast$.  By a theorem of Geck-Hiss \cite[14.4]{cabanesenguehard}, we know $d^1\mathcal{E}(L,s)$ forms a basic set for $\mathcal{E}_\ell(L,s)$, where $d^1$ is the function that restricts a class function to $\ell'$ elements.  (Note here that $\ell$ is good for $\bg{L}$ and $\ell\nmid [\Zent(\bG{L})^F:\Zent^\circ(\bg{L})^F] $ by \cite[3.3]{CabanesEnguehard99}.)  In particular, $d^1\psi$ is an integral linear combination of members of $d^1\mathcal{E}(L,\ell')$.

Note that $d^1$ and $R_L^G$ commute, see \cite[7.5]{dignemichel}.  Further, the constituents of $R_L^G(\psi)$ lie in the same block if and only if the same is true for $d^1R_L^G(\psi)=R_L^G(d^1\psi)$, since this is a sum of Brauer characters in the same blocks as the constituents of $R_L^G(\psi)$. 

Let $\zeta\in\mathcal{E}(L,\ell')$ such that $d^1\zeta$ appears in the decomposition of $d^1\psi$.  The character $\zeta$ must be in the same block $b$ as $\psi$, and we have $R_L^G(\zeta)\subseteq\Z\irr(b^G)$ from above.  Then the irreducible constituents of $R_L^G(d^1\zeta)=d^1 R_L^G(\zeta)$ must lie in the block $b^G$ as well. Since this is true for every such $\zeta$, it is also true for $R_L^G(d^1\psi)$, and hence we must have $R_L^G(\psi)\subseteq\Z\irr(b^G)$.
 \end{proof}
The following can be seen from \cite[Theorems A and B]{KessarMalle15} or \cite[4.1]{CabanesEnguehard99}, by considering the case $\bg{G}=\bg{L}$ and $e=1$. 
\begin{lemma}\label{lem:unique cusp}
Let $\bg{L}$ be a split Levi subgroup of a reductive algebraic group $\bg{G}$ and let $\lambda$ be a cuspidal character of $L:=\bg{L}^F$.  Let $\ell$ be a prime dividing $q-1$ such that $\ell\geq 5$ and $\ell\geq 7$ if $\bg{G}$ has a component of type ${\sf E}_8$.   Suppose $\la\in\mathcal{E}(L,\ell')$ and let $b$ be the $\ell$-block of $L$ containing $\la$.  Then $\la$ is the unique member of $\irr(b)\cap \mathcal{E}(L,\ell')$.  If $(L',\la')$ is another cuspidal pair such that $\la'\in\mathcal{E}(L',\ell')$ and $b_G(L,\la)=b_G(L',\la')$, then $(L,\la)$ is $G$-conjugate to $(L', \la')$.
\end{lemma}

\subsection{First Steps: The Local Side}\label{sec:local}

For the remainder of the section, we will be interested in the situation that $L=\Cent_G(\Zent(L)_\ell)$. We begin by recording two useful consequences of this assumption.

\begin{lemma}\label{lem:uniqueblock}
	Assume that $L=\Cent_G(\Zent(L)_\ell)$ and let $c\in \Bl(L)$ with defect group $D:=D(c)$. Then \begin{enumerate}[label=(\alph*)]
		\item $\Cent_G(D)\leq \Cent_G(\Zent(L)_\ell)=L$
	
	\item If $N'$ is a subgroup of $N=\NNN_{\bg{G}}(\bg{L})^F$ containing $L$, then $c^{N'}$ is defined and is the unique $\ell$-block of $N'$ covering $c$.

\end{enumerate} \end{lemma}
\begin{proof} The first point comes from the fact that $D$ contains any normal $\ell$-subgroup of $L$. For the second point, let $c'\in\mathrm{Bl}(N'|c)$ be a block of $N'$ covering $c$.  Then we may find a defect group $D({c'})$ for $ c'$ such that $D\leq D({c'})$.  Using (a) we know $\Cent_{N'}(D({c'}))\leq \Cent_N(D({c'}))\leq \Cent_N(D)\leq L$. Then by \cite[9.20]{NavarroBlocks}, it follows that $c'$ is regular with respect to $N'$, and hence $c^{N'}$ is defined and $c'=c^{N'}$ by \cite[9.19]{NavarroBlocks}.
\end{proof}

We continue with $(L, \la)$ as in the situation of Lemma~\ref{lem:unique cusp}.
Recall our notation $B:=b_G(L,\la)$ and $b:=b_L(\la)$ with $\la\in\mathcal{E}(L,\ell')$.  Further, recall that we let $\wt{b}\in\Bl(N\mid b)$, and hence $\wt{b}$ is the unique $\ell$-block of $N$ above $b$, by \prettyref{lem:uniqueblock}(b).

\begin{lemma}\label{lem:normalizerD}
Let $\ell$ be a prime dividing $q-1$ and not dividing $[\Zent(\bg{G})^F:\Zent^\circ(\bg{G})^F]$, such that $\ell\geq 5$ and further $\ell\geq 7$ if $\bg{G}$ has a component of type ${\sf E}_8$.  Let $D:=D(B)$. Then the group $N=\NNN_{\bg{G}}(\bg{L})^F$ contains $\NNN_G\left(D\right)$ and is $\aut([G,G])_{B,D}$-stable.  
\end{lemma}
\begin{proof}
We know from \cite[4.16]{CabanesEnguehard99} that $D$ has a unique maximal abelian normal subgroup, $Z$, such that $\NNN_G(\Cent^\circ_{\bg{G}}(Z))\leq N$ and that the extension $1\rightarrow Z\rightarrow D\rightarrow D/Z\rightarrow 1$ is split.  Hence $\NNN_G(D)\leq \NNN_G(Z)\leq \NNN_G(\Cent^\circ_{\bg{G}}(Z))\leq N$. The second statement follows from arguing as in the fifth paragraph of the proof of \cite[5.1]{broughspath}.
\end{proof}
\begin{lemma}\label{lem:samedefect}
 Keep the assumptions of \prettyref{lem:normalizerD}.  The defect groups may be chosen so that $D({\wt{b}})=D(B)$.
\end{lemma}
\begin{proof}
From \prettyref{lem:uniqueblock}(b) and \prettyref{prop:HCblock}, we have $\wt{b}=b^N$ and $B=b^G=\wt{b}^G$.  Then by \cite[4.13]{NavarroBlocks}, we may choose the defect groups such that $D(b)\leq D({\wt{b}})\leq D(B)$, so $\Cent_G(D({\wt{b}}))\leq \Cent_G(D(b))\leq L\leq N$ by \prettyref{lem:uniqueblock}(a).  Then $\wt{b}$ is admissible, and by \cite[9.24]{NavarroBlocks} and \prettyref{lem:normalizerD}, $D({\wt{b}})=D(B)\cap N=D(B)$.
\end{proof}


We will write $\cusp({L})$ and $\cusp(b)$ for the set of irreducible cuspidal
characters of $L$ and in the $\ell$-block $b$, respectively.  The next lemma describes the members of $\irr(\wt{b})$.

\begin{lemma}\label{lem:blockN}
Keep the assumptions of \prettyref{lem:normalizerD}. Then \[\irr(\wt{b})=\{\ind_{N_\psi}^N (\wt{\psi}\eta) \mid\psi\in\cusp(b); \eta\in\irr(N_\psi/L)\},\] where for each $\psi\in\cusp(b)$, $\wt{\psi}$ is a fixed extension of $\psi$ to $N_\psi$.
\end{lemma}
\begin{proof}
For $\psi\in \cusp(L)$, we know by \cite{geck93} and \cite[8.6]{Lusztig84} that $\psi$ extends to some $\wt{\psi}\in\irr(N_\psi)$.  Then Gallagher's theorem \cite[6.17]{isaacs} implies that the characters of the form $\wt{\psi}\eta$, where $\eta$ ranges through all members of $\irr(N_\psi/L)$, are all of the characters of $N_\psi$ above $\psi$.  Clifford theory (see \cite[6.11]{isaacs}) then implies $\ind_{N_\psi}^N (\wt{\psi}\eta)$ is irreducible for $\eta\in\irr(N_\psi/L)$ and that the set $\irr(N|\psi)$ is comprised of the characters of this form.  Then since $\wt{b}=b^N$ is the unique block of $N$ above $b$ by \prettyref{lem:uniqueblock}(b), we see $\{\ind_{N_\psi}^N (\wt{\psi}\eta) \mid\psi\in\cusp(b), \eta\in\irr(N_\psi/L)\}$ is a subset of $\irr(\wt{b})$.  Here for each $\psi\in\cusp(b)$, we have fixed an extension $\wt{\psi}$ of $\psi$ to $N_\psi$.  

Conversely, if $\varphi\in\irr(\wt{b})$, then the constituents of $\res^N_L\varphi$ lie in $N$-conjugates of $b$, and hence $\res^N_L\varphi$ must contain $\psi^g$ as a constituent for some $\psi \in \irr(b)$ and $g\in N$.  But this means that $\res^N_L\varphi$ also contains $\psi$ as a constituent.  Let $\psi$ lie in the Harish-Chandra series of $L$ indexed by the cuspidal pair $(M,\mu)$.  Then note that $b_M(\mu)^G=b^G$ by \prettyref{prop:HCblock} and the transitivity of Harish-Chandra induction.  Further, applying \cite[4.1]{CabanesEnguehard99} and \prettyref{prop:HCblock} to $b_M(\mu)$, we see that $b_M(\mu)=R_{M_1}^M(b_{M_1}(\mu_1))$ for some cuspidal pair $(M_1, \mu_1)$ of $M$ such that $\mu_1\in\mathcal{E}(M_1,\ell')$.  But then again by transitivity of Harish-Chandra induction, we have $b_G(L,\la)=b_G(M_1,\mu_1)$, and hence $(L,\la)$ is $G$-conjugate to $(M_1,\mu_1)$, by \prettyref{lem:unique cusp}.  Then $L=M$ and $\psi\in\cusp(L)$, completing the proof.
\end{proof}

\subsection{Height-Zero Characters and the Map $\Omega$ }\label{sec:hz}
Keep the notation and assumptions from the previous section, and let $\irr(B, L)$ denote the subset of $\irr(B)$ comprised of characters of the form $R_L^G(\psi)_\eta$ for $\psi\in\cusp(b)$ and $\eta\in \irr(W(\psi))$, where $W(\psi):=N_\psi/L$.  Recall that for each $\psi\in\cusp(b)$, we have fixed an extension $\wt{\psi}$ of $\psi$ to $N_\psi$. Recall that $\wt{b}$ is the unique $\ell$-block of $N$ above $b$.

\begin{definition}\label{def:Omega}
Let $\Omega\colon \irr(B, L)\rightarrow \irr(\wt{b})$ be defined via
\begin{equation}\label{eq:defOmega}
\Omega(R_L^G(\psi)_\eta)= \ind_{N_\psi}^N (\wt{\psi}\eta)
\end{equation}   for each $\psi\in\cusp(b)$ and $\eta\in \irr(W(\psi))$ (see Lemma~\ref{lem:blockN}).
\end{definition}

In this section, we aim to show that $\Omega$ induces a bijection
$\Omega\colon \irr_0(B)\rightarrow \irr_0(\wt{b})$, where we write $\irr_0(c)$ for the set of height-zero characters of a block $c$.  Recall that for a finite group $H$ and $c\in\Bl(H)$, a character $\chi$ in $\irr(c)$ satisfies $\chi\in\irr_0(c)$ if and only if $\chi(1)_\ell=|H|_\ell/|D(c)|$.  That is, the $\ell$-part of $\chi(1)$ is as small as possible.  Let $\chi\in\irr_0(B)$ and write $\chi=R_L^G(\psi)_\eta$, so $\Omega(\chi)= \ind_{N_\psi}^N (\wt{\psi}\eta)$.

The next lemma describes $\irr_0(\wt{b})$.

\begin{lemma}\label{lem:localht0}
 Let $\psi\in\cusp(b)$ and $\eta\in\irr(W(\psi))$.  Then $\ind_{N_\psi}^N (\wt{\psi}\eta)\in\irr_0(\wt{b})$ if and only if all of the following hold:
\begin{itemize}
\item $\psi\in\irr_0(b)$;
\item $\eta(1)_\ell=1$; and
\item $[N_b:N_\psi]_\ell=1$.
\end{itemize}
\end{lemma}
\begin{proof}
Let $\chi'$ denote the character $\ind_{N_\psi}^N (\wt{\psi}\eta)$ and write $\tau$ for the irreducible character $\tau:=\ind_{N_\psi}^{N_b} (\wt{\psi}\eta)$ of $N_b$, so that $\ind_{N_b}^N\tau=\chi'$.  Write $b_{N_\psi}$ and $b_{N_b}$ for the blocks of $\wt{\psi}$ and $\tau$, respectively, which are the unique blocks of $N_\psi$ and $N_b$, respectively, above $b$ by \prettyref{lem:uniqueblock}(b).  
By \cite[6.2]{NavarroBlocks}, we have $b_{N_b}=(b_{N_\psi})^{N_b}$ and $\wt{b}=(b_{N_\psi})^N=(b_{N_b})^N$.  By \cite[9.14]{NavarroBlocks}, the defect groups of $b_{N_b}$ and $\wt{b}$ are the same and the height of $\tau$ is the same as the height of $\chi'$.  So $\chi'$ is of height zero if and only if $\tau$ is, and hence it suffices to show that $\tau$ is of height zero if and only if the claimed conditions hold.

Now, by \cite[9.17]{NavarroBlocks},  \prettyref{lem:uniqueblock}(b) implies that \[|D({\wt{b}})|=|D({b_{N_b}})|=|D(b)|\cdot[N_b:L]_\ell\quad \hbox{ and} \quad |D({b_{N_\psi}})|=|D(b)|\cdot[N_\psi :L]_\ell,\] so 
\[|D({b_{N_b}})|=|D({b_{N_\psi}})|\cdot [N_b:N_\psi]_\ell.\]

Then $\tau$ has height zero if and only if $\tau(1)_\ell=[N_b:D({b_{N_b}})]_\ell=\frac{|N_b|_\ell}{|D(b)|\cdot [N_b:L]_\ell}=[L:D(b)]_\ell$.  
But this happens if and only if 
\[\eta(1)_\ell \cdot\psi(1)_\ell \cdot [N_b:N_\psi]_\ell=[L:D(b)]_\ell.\]  
Hence, we immediately see that if the stated conditions hold, then $\tau$ has height zero.  Now assume conversely that $\tau$ has height zero.  
Then \[\psi(1)_\ell\cdot\eta(1)_\ell=[N_\psi:D({b_{N_b}})]_\ell=\frac{|N_\psi|_\ell}{|D({b_{N_\psi}})|\cdot[N_b:N_\psi]_\ell}\leq [N_\psi:D({b_{N_\psi}})]_\ell.\] 
But $\wt{\psi}\eta\in\irr(b_{N_{\psi}})$ implies $\psi(1)_\ell\cdot\eta(1)_\ell\geq [N_\psi:D({b_{N_\psi}})]_\ell,$ so it must be that $\wt{\psi}\eta\in\irr_0(b_{N_\psi})$ and $[N_b:N_\psi]_\ell=1$.  
Then since $\wt{\psi}(1)_\ell\cdot\eta(1)_\ell\geq \wt{\psi}(1)_\ell\geq [N_\psi:D({b_{N_\psi}})]_\ell,$ we also have $\wt{\psi}$  is of height zero and $\eta(1)_\ell=1$.   Finally, \cite[9.18]{NavarroBlocks} now implies that $\psi\in\irr_0(b)$ by applying \prettyref{lem:uniqueblock}(b) again.
\end{proof}

Next we describe the height-zero characters in $\irr(B,L)$.

\begin{lemma}\label{lem:globalht0}
Assume the conditions on $\ell$ from \prettyref{lem:normalizerD}. Let $\psi\in\cusp(b)$ and $\eta\in\irr(W(\psi))$. Then the character $\chi=R_L^G(\psi)_\eta\in\irr(B,L)$ has height zero if and only if all of the following hold:
\begin{itemize}
\item $\psi\in\irr_0(b)$;
\item $\eta(1)_\ell=1$; and
\item $[N_b:N_\psi]_\ell=1$.
\end{itemize}
\end{lemma}
\begin{proof}
By \cite[(7.2)]{MalleSpathMcKay2} or \cite[4.2]{malleheightzero}, we have 
\[\chi(1)=[G:L]_{p'}D_\chi(q)\psi(1)\]
where $D_\chi$ is a rational function with numerator and denominator prime to $X-1$ and $D_\chi(1)=\eta(1)/|W(\psi)|$.  Further, by \cite[6.3]{malleheightzero}, $D_\chi(q)/D_\chi(1)\equiv 1\mod \ell$ and is a rational number with numerator and denominator prime to $\ell$.  Hence we have
\[\chi(1)_\ell=[G:L]_{\ell}\cdot\eta(1)_\ell\cdot\psi(1)_\ell/|W(\psi)|_\ell=[G:L]_{\ell}\cdot\eta(1)_\ell\cdot \psi(1)_\ell/[N_\psi :L]_\ell=[G:N_\psi]_\ell\cdot\eta(1)_\ell\cdot\psi(1)_\ell.\]

 Now, applying \prettyref{lem:samedefect}, we see $\chi$ has height zero if and only if $\chi(1)_\ell=[G:D(B)]_\ell=[G:D(\wt{b})]_\ell$, if and only if 
\[[G:N_\psi]_\ell\cdot\eta(1)_\ell\cdot\psi(1)_\ell=[G:D({\wt{b}})]_\ell.\] But this occurs if and only if  
\[\eta(1)_\ell\cdot\psi(1)_\ell=[N_\psi:D({\wt{b}})]_\ell=\frac{|N_\psi|_\ell}{|D(b)|[N_b:L]_\ell}=\frac{[L:D(b)]_\ell}{[N_b:N_\psi]_\ell}.\]  Here the second equality holds by the proof of \prettyref{lem:localht0}.  

Hence we see that if the stated conditions hold, then $\chi$ has height zero.  Conversely, suppose that $\chi$ has height zero.  Then we must have $\eta(1)_\ell=1$, since $\chi(1)_\ell\geq \chi'(1)_\ell$, where $\chi'=R_L^G(\psi)_{1_{W(\psi)}}$, which lies in the same block by \prettyref{prop:HCblock}.  Then we have
\[[L:D(b)]_\ell\leq \psi(1)_\ell=\frac{[L:D(b)]_\ell}{[N_b:N_\psi]_\ell}\leq [L:D(b)]_\ell,\] implying equality throughout, so $\psi\in\irr_0(b)$ and  $[N_b:N_\psi]_\ell=1$.
\end{proof}

Next, we aim to show that $\irr_0(B)$ is exactly the set of height-zero characters in $\irr(B,L)$.  The following result is key in this direction.  

\begin{proposition}\label{prop:ht0cusp} Keep $(L,\lambda)$ as in Lemma~\ref{lem:unique cusp} with $B=b_G(L,\lambda)$ for $\ell$ a prime dividing $q-1$, not dividing $6[\Zent(\bg G)^F:\Zent^\circ(\bg G)^F]$ and $\geq 7$ if $\bg G$ has components of type ${\sf E}_8$.
If $\irr_0(B)$ contains a cuspidal character, then ${L}= {G}$.
\end{proposition}
\begin{proof} 
Let $\chi\in\Irr_0(B) $ be assumed to be cuspidal. Denote $s\in\bg{L}^*{}^F$ semi-simple and $\ell '$ such that $\la\in\cusp (L)\cap \E (\bg{L}^F,s)$. Then all components of $R_L^G(\la)$ are in $\E(G,s)$. So $B$ meets $\E(G,s)$ and is therefore included in  $\E_\ell (G ,s)$ (see \cite[9.12]{cabanesenguehard}). So we have $\chi\in\E (G ,st)$ where  $t\in \Cent_{\bg{G}^*}(s)^F_\ell$. We know that $\Cent^\circ_{{\bg{G}}^*}(t)$ is a Levi subgroup (see \cite[13.16(ii)]{cabanesenguehard}), let ${\bg G}(t)$ be an $F$-stable Levi subgroup of $\bg G$ in duality with it. Denote by $\hat{t}$ the linear character of ${\bg G}(t)^F$ associated with $t$ by duality (see \cite[13.30]{dignemichel}{}). For an $F$-stable Levi subgroup $\bg M$ of $\bg G$, we write $R_{ \bg M}^{\bg G}$ to denote Deligne-Lusztig induction. If $\bg L$ is a split Levi subgroup of $(\bg G,F)$ then $R_{ \bg L}^{\bg G}$ coincides with Harish-Chandra induction previously denoted by $R_L^G$. 

(a) Assume $t$ is central in $\bg{G}^*$, so that ${\bg G}(t)={\bg G}$. Then $\hat t^{-1}\chi\in \Irr (B)$ since $\hat t^{-1}$ has $\ell$-order. Moreover $\hat t^{-1}\chi\in \E(G ,s)$ by \cite[13.30]{dignemichel}. Using now the description of $\irr(B)\cap\E (G,\ell ')$ (see \cite[4.1(b)]{CabanesEnguehard99}), $\hat t^{-1}\chi$ is a component of $R^{\bg{G}}_{\bg{L}}\la$. Then $\chi$ is a component of $\hat t R_{\bg{L}}^{\bg{G}} (\la)=R_{\bg{L}}^{\bg{G}} (\res^G_L(\hat t)\la)$ and cuspidality implies that ${\bg{L}}={\bg{G}}$.

\medskip

Let us now use the decomposition ${\bg{G}} ={\bg{G}}_{\bf a}{\bg{G}}_{\bf b}$ from \cite[22.4]{cabanesenguehard}. In our case, this means that ${\bg{G}}_{\bf a}$ is generated by $\Zent^\circ ({\bg{G}})$ and the $F$-stable components ${\bg{G}}_i$ of $[{\bg{G}} ,{\bg{G}}]$ such that ${\bg{G}}_i^F$ is of type ${\sf A}_{n_i}(q^{m_i})$ (untwisted) with $m_i,n_i\geq 1$, while ${\bg{G}}_{\bf b}$ is generated by the other components of $[{\bg{G}},{\bg{G}}]$.

\medskip

(b) Assume ${\bg{G}} ={\bg{G}}_{\bf b}$. Then any non central $\ell$-element $t\in {\bg{G}}^*{}^F$ is such that $\Cent_{{\bg{G}}^*}(t)$ embeds in a proper 1-split Levi subgroup ${\bg{C}}^*$ of ${\bg{G}}^*$ thanks to \cite[13.19 and 22.2]{cabanesenguehard}. On the other hand $\chi =R_{\bg{C}}^{\bg{G}}(\chi ')$ for some $\chi '\in\E({\bg{C}}^F,st)$ by \cite[13.25(ii)]{dignemichel}. This contradicts cuspidality, so indeed $t$ is central and case (a) gives our claim.

(c) Assume ${\bg{G}} ={\bg{G}}_{\bf a}$. Let  ${\bg{T}}^*:=\Cent^\circ_{{\bg{G}}^*}(st)$. This is a Levi subgroup and it is included in no proper $1$-split Levi while having a unipotent cuspidal character corresponding with $\chi$ by Jordan decomposition (see for instance \cite[1.10]{CabanesEnguehard99}). In type untwisted $\sf A$, only tori can have a unipotent cuspidal character (hence trivial). So this implies that ${\bg{T}}^*$ is a Coxeter torus of ${\bg{G}}^*$ and $\chi$ is a component of $ R_{\bg{T}}^{\bg{G}}(\wh{st})$. In particular $$\chi (1)_\ell =|{\bg{G}}^F|_\ell/|\Cent_{{\bg{G}}^*}(st)^F|_\ell\eqno(\dagger)$$ by \cite[13.24]{dignemichel}. On the other hand, by the main theorem of \cite{BonnafeDatRouquier} the block $B$ as an algebra over a finite extension of $\Z_\ell$ is Morita equivalent to a block $B_M$ of a subgroup $M$ of $G$ where $ {\bg{C}}^F\lhd M$ with ${\bg{C}}^*=\Cent^\circ_{{\bg{G}}^*}(s)$ and $M/{\bg{C}}^F\cong \Cent_{{\bg{G}}^*}(s)^F/\Cent^\circ_{{\bg{G}}^*}(s)^F$ an abelian group with order prime to $\ell$ (see for instance \cite[13.15(i)]{dignemichel}). Moreover $B_M$ covers a block $B_C$ of ${\bg{C}}^F$ with $\irr(B_C)\subseteq \E_\ell (\bg{C}^F,s)$. Then $\hat s^{-1}B_C$ is a unipotent block. We have ${\bg{C}}={\bg{C}}_{\bf a}$ (see comment after \cite[2.3]{CabanesEnguehard94}) so there is only one $\bg{C}^F$-conjugacy class of unipotent cuspidal pairs $(\bg{L}_C,\la_C)$ in $\bg{C}^F$ by \cite[3.3(i)]{CabanesEnguehard94}. Therefore $\hat s^{-1}B_C$ is the principal block of $\bg{C}^F$, so both $B_C$ and $B_M$ have maximal defect. So height zero characters of $B_M$ have degree prime to $\ell$. Now the Morita equivalence and the fact that $\chi$ has height zero imply that $$\chi(1)_\ell = [{\bg{G}}^F:M]_\ell =[{\bg{G}}^F:{\bg{C}}^F]_\ell \ . \eqno(*)$$ Combining ($\dagger$) and (*) implies that $[{\bg{C}}^F:{\bg{T}}^F]_\ell =|A_{{\bg{G}}^*}(st)^F|_\ell$ where we use the standard notation $A_{{\bg{G}}^*}(x)=\Cent_{{\bg{G}}^*}(x)/\Cent^\circ_{{\bg{G}}^*}(x)$ for $x\in {{\bg{G}}^*}$. One has $$|A_{{\bg{G}}^*}(st)^F|_\ell = |A_{{\bg{C}}^*}(t)^F|_\ell\ ,$$ again because $A_{{\bg{G}}^*}(s)$ is an $\ell '$-group. Using now \cite[13.14(ii)]{dignemichel}, we get that $$[{\bg{C}}^F:{\bg{T}}^F]_\ell \text{ divides } |(\Zent({\bg{C}})/\Zent^\circ ({\bg{C}}))^F|.\eqno(**)$$

Let us show that ($**$) implies ${\bg{C}}={\bg{T}}$.
We have ${\bg{C}}^F/{\bg{T}}^F =({\bg{C}}/{\bg{T}})^F$. The variety ${\bg{C}}/{\bg{T}}$
can be seen in the adjoint group of ${\bg{C}}$. Assuming ${\bg{C}}^F={\bg{C}}_{\bg{a}}^F$ has type $\prod_i{\sf A}_{n_i-1}(q_i)$ ($n_i\geq 2$), we get  $$[{\bg{C}}^F:{\bg{T}}^F]_\ell =\prod_i(q_i^{n_i-1}-1)_\ell (q_i^{n_i-2}-1)_\ell \dots (q_i-1)_\ell .$$ Notice that $q_i$ is a power of $q$, hence $\equiv 1$ mod $\ell$. On the other hand $|(\Zent({\bg{C}})/\Zent^\circ ({\bg{C}}))^F|$ is a divisor of the order of the rational fundamental group (see proof of \cite[13.12(iv)]{cabanesenguehard}) which is $\prod_i(q_i-1,n_i)$. So ($**$) implies $$\prod_i(q_i^{n_i-1}-1)_\ell  \dots (q_i-1)_\ell \text{ divides } \prod_i(q_i-1,{n_i}) .$$ For any $i$ involved, $n_i\geq 2$, so the LHS above is a multiple of $\prod_i(q_i-1)_\ell$. This implies on the RHS that for any $i$ we must have $\ell\mid n_i$ and therefore $n_i\geq \ell > 3$. But then the LHS is a multiple of $\prod_{i}\ell (q_i-1)_\ell$ and this can't divide the RHS unless both products are empty. We then get that $\bg{C}$ is a torus, hence equals $\bg{T}$.

We now have $\Cent^\circ_{{\bg{G}}^*}(s)={\bg{T}}^*$, a Coxeter torus of ${\bg{G}}^*$ and therefore all elements of $\E (G,s)$ are cuspidal, again by \cite[1.10]{CabanesEnguehard99}. But the elements of $\irr(B)\cap\E (G,s)$ are the components of $R_{\bg{L}} ^{\bg{G}}(\la)$ as was recalled before. This forces ${\bg{L}} ={\bg{G}}$.

(d) Before going into the general case, let us notice that if $\bg{L}=\bg{G}=\bg{G}_{\bg a}$ or $\bg{L}=\bg{G}=\bg{G}_{\bg b}$ then all elements of $\irr(B)$ are cuspidal. In the first case this is because as said at the beginning of (c) we must have that $\Cent^\circ_{{\bg{G}}^*}(s)$ is a Coxeter torus of ${{\bg{G}}^*}$ but then it is also the case for any $\Cent^\circ_{{\bg{G}}^*}(st)$ with $t=(st)_\ell $  and therefore any element of $\E_\ell (G,s)$ is cuspidal by \cite[1.10]{CabanesEnguehard99}. In the second case, by \cite[2.8]{CabanesEnguehard99} an element of $\irr(B)$ has to be a constituent of some $R_{\bg{G}(t)}^{\bg{G}}(\hat t\chi_t)$ with $t\in \Cent_{\bg{G}^*}(s)_\ell^F$, ${\bg{G}(t)}$ a Levi subgroup of ${\bg{G}}$ in duality with $\Cent_{\bg{G}^*}^\circ(t)$, $\chi_t\in\E({\bg{G}(t)}^F,s)$ and all components of $R_{\bg{G}(t)}^{\bg{G}}(\chi_t)$ are in $\irr(B)\cap\E(G,s)$. By the description of $\irr(B)\cap\E(G,s)$ (see \cite[4.1(b)]{CabanesEnguehard99}), the last requirement implies that $R_{\bg{G}(t)}^{\bg{G}}(\chi_t)$ is a multiple of the cuspidal character $\chi$. This forbids that ${\bg{G}(t)}$, or equivalently $\Cent_{\bg{G}^*}^\circ(t)$, embed in a proper 1-split Levi subgroup. As said in (b), this implies that $t$ is central. Therefore $R_{\bg{G}(t)}^{\bg{G}}(\hat t\chi_t)=\hat tR_{\bg{G}(t)}^{\bg{G}}( \chi_t)$ is a multiple of $\hat t\chi$ hence cuspidal.

(e) Let us now return to the proof of the statement of the proposition. We look at the general case where ${\bg{G}} ={\bg{G}}_{\bf a}{\bg{G}}_{\bf b}$ and $\chi\in\irr_0(B)$ is cuspidal. Let $H:={\bg{G}}_{\bf a}^F{\bg{G}}_{\bf b}^F$, a normal subgroup of ${\bg{G}}^F$ with abelian $\ell '$ factor group (see \cite[22.5(i)]{cabanesenguehard}). It is a group with split BN-pair given by intersection of the one of ${\bg{G}}^F$. The standard parabolic subgroups correspond in the same way with same radical, and therefore $$^*R^H_{M\cap H}\circ \res^G_H = \res^M_{M\cap H}\circ{}^*R^G_M\text{\ \ \ on  \ }\Z\irr(G)$$ for each split Levi subgroup $M=\bg{M}^F$ of $G$, where $^*R$ denotes Harish-Chandra "restriction" (see \cite[Ch. 4]{dignemichel}, \cite[3.11]{cabanesenguehard}). One deduces easily that the restriction of $\chi$ to $H$ is a sum of cuspidal characters. Let us choose $\chi '\in\irr(H\mid \chi)$ and let $B'$ be its block.  Since $G/H$ is an abelian $\ell '$-group, $B$ and $B'$ have a common defect group (see for instance \cite[9.26]{NavarroBlocks}). By Clifford theory, $\chi(1)_\ell =\chi'(1)_\ell $ so  $\chi '\in\irr_0(B')$. Now $H$ is also the quotient of ${\bg{G}}_{\bf a}^F\times {\bg{G}}_{\bf b}^F$ by a central subgroup $H=({\bg{G}}_{\bf a}^F\times {\bg{G}}_{\bf b}^F)/Z$ where $Z\cong \Zent({\bg{G}}_{\bf a}^F)\cap Z ({\bg{G}}_{\bf b}^F)=\Zent({\bg{G}}_{\bf a}\cap {\bg{G}}_{\bf b})^F$ is also $\ell '$ by \cite[22.5(i)]{cabanesenguehard}. The blocks and characters of $H$ can be seen as blocks and characters of ${\bg{G}}_{\bf a}^F\times {\bg{G}}_{\bf b}^F$ with $Z$ in their kernel. This obviously preserves cuspidality. So $B'$ corresponds to a block $B''$ of ${\bg{G}}_{\bf a}^F\times {\bg{G}}_{\bf b}^F$ with a character $\chi ''$, corresponding to $\chi '$, that is of height zero and cuspidal. By (b) and (c) above we have $B''=b_{{\bg{G}}_{\bf a}^F\times {\bg{G}}^F_{\bf b}}({\bg{G}}_{\bf a}^F\times {\bg{G}}_{\bf b}^F ,\la_0)$ for some $\la_0\in\cusp ({{\bg{G}}_{\bf a}^F\times {\bg{G}}^F_{\bf b}})$. Now (d) implies that all characters of $B''$ are cuspidal. Therefore the same is true for $B'$ and $B$ as discussed before. Then all components of $R^{\bg{G}}_{\bg{L}}(\la)$ have to be cuspidal. This clearly implies ${\bg{L}}={\bg{G}}$.
\end{proof}

\begin{remark}\label{rem:ht0cusp}
(a) It is easy to deduce from the above proof that the $\ell$-blocks of the form $B=b_G({G},\la)$ with $\la\in\cusp(G)\cap\E(G,\ell ')$, have only cuspidal characters, i.e. $\irr(B)\subseteq \cusp(G)$.

(b) There are indeed $\ell$-blocks $b_G(L ,\la)$ (notation of \cite{CabanesEnguehard99}) with $\ell$ satisfying the hypotheses of Proposition~\ref{prop:ht0cusp} and $\bg{L}\not= \bg{G}$ but with cuspidal characters in $\irr(b_G(L ,\la))$. An example is as follows. Let $\bg{G}$ be GL$_\ell (\overline\F_q)$ with the ordinary Frobenius endomorphism $F$ such that $\bg{G}^F=\operatorname{GL}_\ell(q)$ and $\ell\mid (q-1)$. Let $\bg{L}$ be the diagonal torus, $L:={\bg L}^F$ and $\la =1_{L}$. Then $b_G(L ,1)$ is the principal block with $\E(G,1)\subseteq\irr(b_G(L ,\la))$ again by \cite[3.3(i)]{CabanesEnguehard94} and in fact $\irr(b_G(L ,\la))=\E_\ell(G,1)$ by \cite[9.12(ii)]{cabanesenguehard}. Let $\bg{T}$ be the Coxeter torus of $\bg{G}$ with $\bg{T}^F$ of order $q^\ell-1$. It is easy to find a regular element in $\bg{T}^F$ of order $(q^\ell -1)_\ell =\ell (q-1)_\ell$ and dually a regular character $\theta\in\irr(\bg{T}^F)$ of multiplicative order a power of $\ell$. Then $R_{\bg{T}}^{\bg{G}}\theta$ is up to a sign a cuspidal character and an element of $\E_\ell (G,1)$. So it is also an element of $\irr(b_G( {L},1_L))$, though not of height zero.

\end{remark}

\begin{corollary}\label{cor:blockB}
Let $\bg{L}$ be a split Levi subgroup of a reductive algebraic group $\bg{G}$ and let $\lambda$ be a cuspidal character of $L:=\bg{L}^F$.  Let $\ell$ be a prime dividing $q-1$ such that $\ell\geq 5$, $\ell\geq 7$ if $\bg{G}$ has a component of type ${\sf E}_8$, and $\ell\nmid [\Zent(\bG{G})^F:\Zent^\circ(\bg{G})^F]$.   Suppose $\la\in\mathcal{E}(L,\ell')$ and let $b=b_L(\la)$ be the $\ell$-block of $L$ containing $\la$ and $B=b^G$ be the $\ell$-block of $G=\bg{G}^F$ containing the irreducible components of $R_L^G(\la)$ (see Proposition~\ref{prop:HCblock}).  Then
{\[\irr_0(B)=\irr_0(b^G)\subseteq\{R_L^G(\psi)_\eta \mid\psi\in\irr_0(b) \hbox{ cuspidal}\ ; \eta\in\irr(W(\psi))\}.\] }
\end{corollary}
\begin{proof}
Let $\chi\in\irr_0(B)$.  Note that by \prettyref{lem:globalht0}, it suffices to show that $\chi$ is in the Harish-Chandra series $R_L^G(\psi)$ for some cuspidal $\psi\in\irr(b)$.  We may assume that $L\neq G$.  

We know $\chi$ must be a constituent of some $R_M^G(\mu)$ for $\mu$ a cuspidal character of a split Levi $M:=\bg{M}^F$ of $G$.   Write $b_M(\mu)$ for the block of $M$ containing $\mu$.  By \prettyref{prop:HCblock}, all constituents of $R_M^G(\mu)$ lie in the same block, $b_{M}(\mu)^G$.  Then it must be that $b_L(\la)^G=B=b_{M}(\mu)^G$.  Further, applying \cite[4.1]{CabanesEnguehard99} and \prettyref{prop:HCblock} to $b_M(\mu)$, we see that $b_M(\mu)=R_{M_1}^M(b_{M_1}(\mu_1))$ for some cuspidal pair $(M_1, \mu_1)$ of $M$ such that $\mu_1\in\mathcal{E}(M_1,\ell')$.  But then by transitivity of Harish-Chandra induction, we have $b_G(L,\la)=b_G(M_1,\mu_1)$, and hence $(L,\la)$ is $G$-conjugate to $(M_1,\mu_1)$, by \prettyref{lem:unique cusp}, and we may assume $L\leq M$. 

 Now, note that the arguments in Lemmas \ref{lem:uniqueblock}(b) and \ref{lem:samedefect} still hold in the case with $N$ replaced with $K:=\NNN_{\bg{G}}(\bg{M})^F$ and $b$ replaced with $b_M(\mu)$, but with the statement $D({\wt{b}})=D(B)$ replaced with $D({\wt{b_M}(\mu)})\leq D(B)$. 
  (Here $\wt{b_{M}}(\mu)$ is the block of $K$ above $b_M(\mu)$.) Then we may argue as in \prettyref{lem:globalht0} to see that $\mu$ must have height zero if $\chi$ does.  Indeed, letting $K_{b_M(\mu)}$ denote the stabilizer of $b_M(\mu)$ in $K$, we have in this case 
  
 \[[M:D({b_M(\mu)})]_\ell\leq \mu(1)_\ell=\frac{[G:D(B)]_\ell}{[G:K_\mu]_\ell}=[K_\mu:D(B)]_\ell\leq[K_\mu:D({\wt{b_M}(\mu)})]_\ell\]
 \[=\frac{|K_\mu|_\ell|M|_\ell}{|D({b_M(\mu)})|_\ell|K_{b_M(\mu)}|_\ell}=\frac{[M:D({b_M(\mu)})]_\ell}{[K_{b_M(\mu)}:K_\mu]_\ell}\leq [M:D({b_M(\mu)})]_\ell,\] implying equality throughout.  But this contradicts \prettyref{prop:ht0cusp} applied to $b_M(\mu)$, so $M=L$ and $\chi$ must be a constituent of $R_L^G(\psi)$ for some cuspidal $\psi\in\irr(b)$, as desired.
\end{proof}

Lemmas \ref{lem:localht0} and \ref{lem:globalht0} and \prettyref{cor:blockB} immediately yield the following.

\begin{corollary}\label{cor:map}
Let $\ell$ be a prime dividing $q-1$ and not dividing $[\Zent(\bG{G})^F:\Zent^\circ(\bg{G})^F]$, such that $\ell\geq 5$ and $\ell\geq 7$ if $\bf{G}$ has a component of type ${\sf E}_8$.  Then the map $\Omega$ (see Definition~\ref{def:Omega}) restricts to a bijection
\[\Omega\colon \irr_0(B)\rightarrow \irr_0(\wt{b})\]
\[R_L^G(\psi)_\eta\mapsto \ind_{N_\psi}^N (\wt{\psi}\eta).\]
\end{corollary}

\begin{remark}\label{rem:Tstabla}
We remark that $N_b=N_\la$, since if $x\in N_b$ then $\la^x\in b$ and lies in $\mathcal{E}(L,\ell')$, and hence $\la^x=\la$ by \prettyref{lem:unique cusp}.
\end{remark}

\section{The Inductive Alperin-McKay Conditions}\label{sec:inductive am}

In this section we give a criterion implying the inductive Alperin-McKay conditions of \cite[7.2]{spathAMreduction} and tailored to simple groups of Lie type. This generalizes the one given in \cite[Sect. 4]{cabanesspath15} which doesn't cover all cases. Instead, to prove \prettyref{thm:iAMtypeC}, we will use the following easy adaptation of \cite[2.4]{broughspath}.

\begin{theorem}	[Brough-Sp\"ath]\label{thm:newAMconds}
Let $S$ be a finite non-abelian simple group and $\ell$ a prime dividing $|S|$.
Let $G$ be the universal covering group of $S$ and assume we have a semi-direct product $\wt G \rtimes E$ with $[\wt G,\wt G]=G\lhd \wt G \rtimes E$ and $\calB\subseteq \Bl (G)$ a $\wt G$-stable subset such that for every $B\in \calB$ the inclusion $(\wt GE)_B\leq (\wt GE)_\calB$ holds.
Assume there exist groups $M\lneq G$ and $\wt M\leq \wt G$ such that
$M=\wt M\cap G$ and $\wt M\geq M \NNN_{\wt G}(D)$, which further satisfy that 
 for every $\ell$-block $B\in \calB$ and some defect group $D$ of $B$, 
 $M$ is $\Aut (G)_{\calB,D}$-stable and $\NNN_G(D)\leq M\lneq G$. 
 Let $\calB'\subseteq \Bl (M)$ be the set of all Brauer correspondents of the $\ell$-blocks in $\calB$. Additionally assume:
	\begin{asslist}
		\item 
		\begin{itemize}[label={\textbullet}]
			\item $\Cent_{\wt G\rtimes E}(G)=\Zent(\wt G)$ and $\wt GE/\Zent(\wt G)\cong \Inn (G)\Aut (G)_{D}$ by the natural map,
			\item any element of $\Irr_0(\calB)$ extends to its stabilizer in $\wt G$,
			\item any element of $\Irr_0(\calB')$ extends to its stabilizer in $\wt M$, 
			\item the group $E$ is abelian.
		\end{itemize}
		\item \label{thm24ii}
		For $\calG:=\Irr ( \wt G\mid \Irr_0(\calB)  )$ and $\calM:=\Irr ( \wt M\mid \Irr_0(\calB')  )$ there exists an $\NNN_{\wt GE}(D)_\calB$-equivariant bijection
		\[\wt{\Omega}:\calG \longrightarrow \calM\]
		with
		\begin{itemize}
			\item $\wt{\Omega}\left( \calG\cap \Irr(\wt G\mid \wt{\nu}) \right)=\calM\cap \Irr(\wt M\mid \wt{\nu})$ for all $\wt{\nu}\in \Irr ( \Zent(\wt G)  )$,
			\item $b_{\wt M}\left( \wt{\Omega}(\wt{\chi}) \right)^{\wt G}=b_{\wt G} (\wt{\chi})$ for all $\wt{\chi}\in \calG$, and
			\item $\wt{\Omega}(\wt{\chi}\wt{\mu})=\wt{\Omega}(\wt{\chi})\wt{\mu}_{\wt M}$ for every $\wt{\mu}\in \Irr(\wt G\mid 1_G)$ and every $\wt{\chi}\in \calG$.
		\end{itemize}
		\item For every $\wt{\chi}\in \calG$ there exists some $\chi_0\in \Irr (G\mid \wt{\chi})$ such that
		\begin{itemize}[label={\textbullet}]
			\item $(\wt G\rtimes E)_{\chi_0}=\wt G_{\chi_0}\rtimes E_{\chi_0}$, and 
			\item $\chi_0$ extends to $G\rtimes E_{\chi_0}$.
		\end{itemize}
		
		\item For every $\wt{\psi}\in \calM$ there exists some $\psi_0\in \Irr (M\mid \wt{\psi})$ such that
		\begin{itemize}[label={\textbullet}]
			\item $O=(\wt G\cap O)\rtimes (E\cap O)$ for $O:=G(\wt G\times E)_{M,\psi_0}$, and 
			\item $\psi_0$ extends to $M(G\rtimes E)_{D,\psi_0}$.
		\end{itemize}
		\item \label{NewIndAmCondv} For every $B\in \calB$ and its $\wt G$-orbit $\wt{B}$ the group $\out (G)_{\wt{B}}$ is abelian.
	\end{asslist}
	Then the inductive Alperin-McKay conditions (see \cite[7.2]{spathAMreduction}) hold for all $\ell$-blocks in $\calB$.
\end{theorem}

\subsection{Criterion With Levi Subgroups}

Here we adapt the conditions from \prettyref{thm:newAMconds} specifically to fit the bijection $\Omega$ from \prettyref{cor:map}.  Throughout this section, let $\bg{G}$ be a simple algebraic group of simply connected type over an algebraic closure of the field with $p$ elements. We assume chosen a Borel subgroup and a maximal torus $\bg{T}\leq \bg B$ and we will denote by $\Phi\supseteq\Delta$ the root system and the basis corresponding to $\bg B$. One recalls the 1-parameter unipotent subgroups $t\mapsto \xx_\alpha (t)$ for $\alpha\in\Phi$ and $t\in\overline{\F}_p$. We let $\bg{X}_{\alpha}:=\bg{x}_{\alpha}(\overline \F_q)$. One defines $F_0$ on $\bg G$ by $F_0(\xx_\alpha (t))=\xx_\alpha (t^p)$. One calls graph automorphisms (omitting to mention $\bg T$ and $\bg B$) the automorphisms of $\bg G$ defined by $\xx_{\epsilon\delta} (t)\mapsto \xx_{{\epsilon\delta '}}(t)$ for $\epsilon \in\{\pm 1\}$, $\delta\in\Delta$ and $\Delta\ni \delta\mapsto\delta '\in\Delta$ an automorphism of the associated Dynkin diagram.

We let $\bg{G}\hookrightarrow\wt{\bg{G}}$ be a regular embedding as in \cite[15.1]{cabanesenguehard}.  In particular, $\wt{\bg{G}}$ is a central product $\wt{\bg{G}}=\Zent(\wt{\bg{G}})\bg{G}$ and both $F_0$ and the graph automorphisms of $\bg G$ extend to $\wt{\bg G}$  (see \cite[Sect. 2.B]{MalleSpathMcKay2}). 
We let $F:=F_0^m\gamma$ where $\gamma$ is a graph automorphism (possibly trivial) and $m\geq 1$. We denote $q=p^m$ so that $\wt{\bg G}$ and ${\bg G}$ are defined over $\F_q$ via $F$. 

We also denote $G:=\bg{G}^F$, $\wt{G}:=\wt{\bg{G}}^F$ and let $E$ be the subgroup of $\Aut(\wt G)$ generated by the restrictions of $F_0$ and the graph automorphisms considered above.

 Let $\ell$ denote a prime not dividing $q$. All blocks considered will be $\ell$-blocks.

Let $\bg{L} ={\bg T}\spann<{\bg X}_\al\mid\al\in\Phi '>$ be a standard Levi subgroup of $\bg{G}$ associated with $\Phi':=\Phi\cap\R \Delta '$ for some $\Delta '\subseteq\Delta$ which we assume $\gamma$-stable, so that $\bg L$ is $F$-stable. Now let $\wt{\bg{L}}:=\bg{L}\Zent(\wt{\bg{G}})$ be the corresponding split Levi subgroup of $\wt{\bg{G}}$.  Write $L:=\bg{L}^F$ and $\wt{L}:=\wt{\bg{L}}^F$ for the resulting Levi subgroups of $G$ and $\wt{G}$, respectively.  Write $T:=\bg{T}^F$ and $\wt{T}:=\wt{\bg{T}}^F$. Let $N:=\NNN_{\bg{G}}({\bg{L}})^F$ and $\wt{N}:=\NNN_{\wt{\bg{G}}}(\wt{\bg{L}})^F$, and note that \[\wt{N}=\wt{L}N=\wt{T}N\] by a standard application of Lang's theorem.

In this section, we aim to prove the following:
\begin{proposition}\label{prop:ourconds}
Let $G=\bg{G}^F$ as above and assume that $G$ is the universal covering group of the non-abelian simple group $G/\Zent(G)$. Let $\ell$ be a prime dividing $q-1$ but not dividing $6|\Zent(G)|$, and further assume $\ell\geq 7$ if $\bg{G}$ is of type ${\sf E}_8$.  Let $B\in\Bl(G)$ corresponding to a cuspidal pair of $L$, where $L$ is $E$-stable.
Assume that $E$ is cyclic and that for $\wt{L}, N,$ and $\wt{N}$ as above, there is an $NE$-stable $\wt{L}$-transversal $\mathbb{T}\subseteq\cusp(L)$ such that the following hold:
\begin{enumerate}[label=(\arabic*)]
\item There is an $NE$-equivariant extension map (see Definition~\ref{ExtMap}) with respect to $L\lhd N$ for $\mathbb{T}$. 
\item $R(^t\la)\leq \ker(\delta_{\la,t})$ for all $\la\in \mathbb{T}$ and $t\in\wt{T}$ with the notation from \cite[Sect. 4]{MalleSpathMcKay2}.
\item For every $\wt{\chi}\in \irr( \wt G\mid\irr_0(B))$, there exists some $\chi_0\in\irr(G\mid\wt{\chi})$ such that
 $(\wt{G}\rtimes E)_{\chi_0}=\wt{G}_{\chi_0}\rtimes E_{\chi_0}$.

\item $\out(G)_{\wt B}$ is abelian for the $\wt G$-orbit $\wt B$ of $B$.
\end{enumerate}
Then the inductive Alperin-McKay conditions hold  for $B$.

\end{proposition}

\begin{remark}\label{starCond}
The condition (3) above is equivalent to the existence of an $E$-stable $\wt G$-transversal in $\Irr_0(B)$. Indeed, for each $\wt G\rtimes E$-orbit it suffices to select one $\chi_0$ as in the condition and take the images under $E$-action. The stabilizer property will ensure that this is a $\wt G$-transversal. The converse is also easy.

\end{remark}

We begin by recording the following straightforward observation:
\begin{lemma}\label{lem:ht0above}
Suppose $\ell\nmid |\Zent(G)|$.  
Let $B\in\Bl(G)$ and $C\in\Bl(N)$, $\wt B=\Bl(\wt G|B)$ and $\wt C=\Bl(\wt N|b)$. Then $\irr (\wt{G}\mid \irr_0(B) )=\irr_0(\wt B)$ and $\irr (\wt{N}\mid \irr_0(C) )=\irr_0(\wt{C})$.
\end{lemma} 

\begin{proof}
 Note that by \cite[Proposition 10]{lusztig88}, any $\chi\in\irr(G)$ extends to its inertia group $\wt{G}_\chi$ in $\wt{G}$. Since $\ell\nmid [\wt{G}:\Zent(\wt{G})G]$ and $\Zent(\wt{G})G\leq \wt{G}_\chi$ for any $\chi\in\irr(G)$, it suffices to prove the statement for $\Zent(\wt{G})G$ rather than $\wt{G}$. Let $\wt{\chi}\in\irr(\Zent(\wt{G})G|\chi)$ be an extension of $\chi\in\irr(G)$ and let $\wt{D}$ and $D$ be the defect groups of $\wt B$ and $ B$ in $\wt{G}$ and $G$, respectively.  Note that $\Zent(\wt{G})_\ell\leq \wt{D}$ and $\Zent(G)_\ell\leq D$, and hence $\wt{D}$ can be chosen so that $\Zent(\wt{G})_\ell D\leq \wt{D}$.  Further, \cite[9.17]{NavarroBlocks} yields that $|\wt{D}|\leq |\Zent(\wt{G})_\ell D|$.  Then $ [\Zent(\wt{G})G:\wt{D}]_\ell= [\Zent(\wt{G})G:\Zent(\wt{G})D]_\ell=[G:D]_\ell$ and $\wt{\chi}$ has height zero if and only if $\chi$ does, giving the statement in $\wt G$. 
 
 Since $\wt{N}/N=\wt{L}N/N\cong\wt{L}/L\cong\wt{G}/G$ and the defect groups of $\wt{C}$ and $C$ are the same as the corresponding defect groups for $\wt{B}$ and $B$ under the maps constructed in \prettyref{sec:map}, the same arguments show the statement in $\wt N$. 
 \end{proof}

 As in Sections \ref{sec:local} and \ref{sec:hz}, we assume that $\ell$ divides $q-1$ but not $6\cdot [\left(\Zent(\bg{G}):\Zent^\circ(\bg{G})\right)^F]$ and $\ell\geq 7$ if $\bg G$ is of type $\sf E_8$. Then by \cite[4.1]{CabanesEnguehard99} any $\ell$-block of $G=\bg G^F$ is of the type $b_G(L,\la)$ studied before, and the same is true for $\wt{\bg G}^F$. Note that for $D$ a defect group of $G$ and $\wt{D}$ a defect group of $\wt{G}$ such that $D=\wt{D}\cap G$, \prettyref{cor:map} applied independently to $G$ and $\wt G$ then yields bijections 
 \[\Omega\colon \irr_0(G\mid D)\rightarrow \irr_0(N\mid D)\] and 
 \[\wt{\Omega}\colon  \irr_0(\wt{G}\mid \wt{D})\rightarrow \irr_0(\wt{N}\mid \wt{D})\] simultaneously, each preserving Brauer correspondence.  We wish to use information about $\Omega$ to obtain the properties for $\wt{\Omega}$ required in  \prettyref{prop:ourconds}.

Recall from  Definition~\ref{def:Omega} that the construction of $\Omega$ depends on the choice of an extension map $\psi\mapsto \wt\psi$ with respect to the normal inclusions $L\lhd N$ for $\cusp(L)$.

 \begin{lemma} \label{lem:omegapropsabove}
 
 Assume that for any standard Levi subgroup $\wt L$ in $\wt G$, there is an $\wt{N}E$-equivariant extension map $\wt{\Lambda}$ (see Definition~\ref{ExtMap}) for $\cusp(\wt{L})$ with respect to $\wt{L}\lhd\wt{N}$.  Then:
 \begin{enumerate}[label=(\alph*)]
 \item  
   The map $\wt{\Omega}\colon \irr_0(\wt{G}\mid\wt{D})\rightarrow\irr_0(\wt{N}\mid\wt{D})$ described above is $\NNN_{\wt{G}E}(D)$-equivariant.
 \item 
 If the map $\wt{\Lambda}$ satisfies \[\wt{\Lambda}\left(\wt{\la}\cdot\mu|_{\wt{L}}\right)=\wt{\Lambda}(\wt{\la})\cdot\mu|_{\wt{N}_{\wt{\la}}}\]
  for each $\wt{\la}\in\cusp(\wt{L})$ and each $\mu\in\irr(\wt{G}\mid 1_G)$, then $\wt{\Omega}$ satisfies \[\wt{\Omega}(\wt{\chi}{\mu})=\wt{\Omega}(\wt{\chi})\cdot{\mu}|_{\wt{N}}\] for every $\wt{\chi}\in\irr_0({\wt{G}})$ and $\mu\in\irr(\wt{G}\mid 1_G)$.
 \end{enumerate}
 \end{lemma}
 \begin{proof}
Both statements follow directly from our construction of $\wt{\Omega}$, taking into account \cite[4.6, 4.7]{MalleSpathMcKay2} for part (a). Note that, thanks to the $\wt NE$-equivariance of $\wt\Lambda$, the linear character $\delta_{\la ,\sigma}$ in \cite[4.6]{MalleSpathMcKay2} is trivial in the case of an automorphism $\sigma$ induced by an element of $\wt NE$.
 \end{proof}
 
 \begin{lemma}\label{lem:equivextabove}
Assume condition (1)  of \prettyref{prop:ourconds}.  Then there is an $\wt{N}E$-equivariant extension map $\wt{\Lambda}$ for $\cusp(\wt{L})$ with respect to $\wt{L}\lhd \wt{N}$ satisfying
\[\wt{\Lambda}\left(\wt{\la}\cdot\mu|_{\wt{L}}\right)=\wt{\Lambda}(\wt{\la})\cdot\mu|_{\wt{N}_{\wt{\la}}}\] for each $\wt{\la}\in\cusp(\wt{L})$ and each $\mu\in\irr(\wt{G}\mid 1_G)$. \end{lemma}
 \begin{proof}
 Let $\wt{\la}\in\cusp(\wt{L})$ and $\la_0\in\irr(L|\wt{\la})\cap\mathbb{T}$.  Fix an extension $\wt{\la}_0$ of $\la_0$ to $\wt{L}_{\la_0}$ such that $\wt{\la}=\ind_{\wt{L}_{\la_0}}^{\wt{L}}(\wt{\la}_0)$.  Note that since $\mathbb{T}$ is $N$-stable, we have $\wt{L}N_{\wt{\la}_0}=\wt{L}N_{\wt{\la}}$, using Clifford theory. Let $\Lambda$ be the assumed extension map with respect to $L\lhd N$, so that $\Lambda(\la_0)$ is a character of $N_{\la_0}$ extending $\la_0$. Then by \cite[5.8 (a)]{CabanesSpathInductiveMcKayTypeA} or \cite[4.1]{spath10b}, there exists a unique common extension, call it $\varphi$, of $\wt{\la}_0$ and $\Lambda(\la_0)|_{N_{\wt{\la}_0}}$ to $\wt{L}_{\la_0}N_{\wt{\la}_0}$.  Define 
 \[\wt{\Lambda}(\wt{\la}):=\ind_{\wt{L}_{\la_0}N_{\wt{\la}_0}}^{\wt{L}N_{\wt{\la}_0}}(\varphi).\] Then $\wt{\Lambda}(\wt{\la})$ is an extension of $\wt{\la}$ to $\wt{N}_{\wt{\la}}=\wt{L}N_{\wt{\la}}=\wt{L}N_{\wt{\la}_0}$.  This defines an extension map $\wt{\Lambda}$, which by construction is $NE$-equivariant. The map $\wt \Lambda$ is $\wt L$-equivariant, hence $\wt N$-equivariant since $\wt N=\wt L N$. The required equation $
 \wt{\Lambda}\left(\wt{\la}\cdot\mu|_{\wt{L}}\right)=\wt{\Lambda}(\wt{\la})\cdot\mu|_{\wt{N}_{\wt{\la}}}$ holds since $\wt{\Lambda}(\wt{\la}\cdot\mu|_{\wt{L}})$ is constructed using $\la_0\in \mathbb{T} \cap \Irr(L|{\wt{\la}\cdot\mu|_{\wt{L}}})$ and the common extension of $\wt{\la}_0 \mu|_{\wt{L}_{\la_0}} $ and $\Lambda(\la_0)|_{N_{\wt{\la}_0}}$.
  \end{proof}
 
We are now ready to prove \prettyref{prop:ourconds}.
 
 \begin{proof}[Proof of \prettyref{prop:ourconds}] We check that the assumptions of Theorem~\ref{thm:newAMconds} are satisfied. We have $S=G/\Zent(G)$ of which $G$ is a universal covering by assumption and on which $\wt G\rtimes E$ induces the whole automorphism group by \cite[2.5.1]{gorensteinlyonssolomonIII}. We also have the stabilizer part of assumption 3.1(iii) by 3.2(3). Note that the extendibility conditions of 3.1(iii) and 3.1(iv), along with the remainder of condition 3.1(i) are ensured by the assumption that $E$ and hence $M(GE)_{D,\psi_0}/M$ are cyclic.
 	
Note that a $\wt{G}$-orbit $\wt{B}$ containing $B=b_G(L,\la)\in\Bl(G)$ is composed of blocks $b_G(L,\la')$ for other $\la'\in\cusp(L)\cap\mathcal{E}(L,\ell')$, and hence $N=\NNN_{\bg G}(\bg L)^F$ contains $\NNN_G(D_B)$ for each $B\in\wt{B}$, applying \prettyref{lem:normalizerD}.  Taking $M:=N$ and $\wt{M}:=\wt{N}$, we see using Lemmas \ref{lem:ht0above}, \ref{lem:omegapropsabove}, and \ref{lem:equivextabove}, together with our assumptions, that assumption (ii) of \prettyref{thm:newAMconds} holds.

Our map $\Omega$ is built with the same method as for the bijection in \cite[5.2]{MalleSpathMcKay2}. 
The arguments from there can be applied thanks to assumptions 3.2(1) - (2) and show that $\Omega$ is $\NNN_{\wt{G}E}(D)$-equivariant. 

In order to now ensure  condition (v) of \prettyref{thm:newAMconds} we apply the considerations from the proof of 
\cite[5.3]{MalleSpathMcKay2}:
 Let $\wt{\psi}\in\irr_0(\wt{M})$.  As in the proof of \cite[4.3]{cabanesspathtypeC}, it suffices to show that $(\wt{M}\wh{M})_{\psi_0}=\wt{M}_{\psi_0}\wh{M}_{\psi_0}$, where $\wh{M}:=NE$ and $\psi_0$ is a suitable member of $\irr(M\mid\wt{\psi})$.  But this follows by taking $\psi_0:=\Omega(\chi_0)$, where $\chi_0\in\irr(G\mid\wt{\chi})$ satisfies assumption 3.2(3) and $\wt{\psi}=\wt{\Omega}(\wt{\chi})$.
\end{proof}
\section{Extending Cuspidal Characters in Type ${\sf C}_l$} \label{sec:ext_C}
Our main task is now to verify in $G={\rm Sp}_{2l}(q)$ the existence of $\Bbb T$ satisfying assumption (1) of Proposition~\ref{prop:ourconds}, namely we construct an $N E$-stable $\wt L$-transversal $\bb{T}\subseteq \cusp(L)$ and an extension map with respect to $L\lhd N$ for $\bb{T}$. 

 We will check this via an application of the following  
criterion,
which is based on \cite[4.2]{broughspath}. It will be applied in a case where $K=K_0$  but we show the slightly stronger statement for future reference.

\begin{proposition} \label{cor_tool}
	Let $K\lhd M$ and $K_0\lhd M$ with $K_0\leq K$ be finite groups, and let $E$ be a group acting on $M$, stabilizing $K$ and $K_0$. In addition, let $\mathbb K\subseteq \Irr(K)$ be $ME$-stable. Assume 
	\begin{asslist}
		\item \label{cor32i} $K=\Zent(K)K_0$;
		\item  there exists some $E$-stable group $V\leq M$ such that
		\begin{enumerate}
			\item[(ii.1)] 	\label{cor23ii1}
			$M=KV$ and $H:=V\cap K\leq \Zent(K)$; and 
			\item[(ii.2)] \label{cor23ii2} \label{prop23ii2}
			there exists some $VE$-equivariant extension map $\Lambda_0$ with respect to $H\lhd V$ for $\bigcup_{\la\in \mathbb K} \Irr(H\mid  \la)$;
		\end{enumerate}
		\item \label{cor_toolext} denoting $\epsilon\colon V\to V/H$ the canonical surjection,
		there exists an $\epsilon(V)E$-equivariant extension map $\Lambda_\epsilon$ with respect to $K_0\lhd K_0\rtimes \epsilon(V_{})$ for the set $\bigcup_{\la\in \mathbb K} \Irr(K_0\mid  \la)$. 
	\end{asslist}
	Then there exists an $ME$-equivariant extension map with respect to $K\lhd M$ for $\mathbb K$.
\end{proposition}

\begin{proof}
	By (ii.1) it is sufficient to construct a $VE$-equivariant extension map \wrt $K\lhd M$ for $\Bbb{K}$. 
	
	Let $\la\in \Bbb{K}$. Proposition 4.2 in \cite{broughspath} defines an extension $\wt \la$ of $\la$ to $M_\la$ in the following way: the character $\zeta\in\Irr(H\mid \la)$ has the extension $\wt\zeta:=\Lambda_0(\zeta)$ and $\la_0:=\res_{K_0}^K \la$
	extends to $\Lambda_\epsilon(\la_0)\in\Irr( K_0\rtimes \epsilon(V_{\la_0})$. Those two extensions are used to  define $\wt \cD\colon M_\la\to\operatorname{GL}_{\la(1)}(\C)$ via the equation 	
	\begin{align}\label{wtcD} 
	\wt	\cD( kv)&= \wt \zeta(v)\cD' (\epsilon(v)) \mathcal D(k) \text{ for every }
	k\in K \text{ and }v\in V_\la ,\end{align}
	where $\mathcal D$ is a linear representation of $K$ affording $\la$, and $\cD'$ is a linear representation of $K_0\rtimes \epsilon(V)_{\lambda_0}$ extending $ \res^K_{K_0}\cD$ and affording $\Lambda_\epsilon(\lambda_0)$.
	
	We obtain an extension map $\Lambda$ with respect to $K\lhd KV$ for $\mathbb K$ given by $\la\mapsto \operatorname{Tr}\circ \wt \cD$. Note that this map is well-defined. Since the extension maps $\Lambda_0$ and $\Lambda_\epsilon$ are $VE$-equivariant, one checks easily  
	using the above formula that $\Lambda$ is $VE$-equivariant.
\end{proof}

\subsection{The structure of $L$ and $N$ in type $\sf C$}\label{subsec_L_C}

We now concentrate on finite quasi-simple groups of Lie type $\sf C$. Though the structure of split Levi subgroups in symplectic groups is a direct product easily dealt with, their normalizers don't equal the corresponding wreath products, so the problem of character extensions requires some special care. 

  For a positive integer $i$ let $\underline{i}:=\{1,\ldots, i \}$.

\begin{notation}\label{not:typeC}
	Let $\bg{G}={\mathrm{Sp}}_{2l}(\overline\F_q)$ be a simply connected simple group of type ${\sf C}_l$ over the field $\overline\F_q$. Let $\bg T$ be the diagonal torus and $\bg B$ be the upper triangular Borel subgroup of $\bg G$. Let $\Phi$ be the $\bg T$-roots of $\bg{G}$ given as $\{ 2e_i,  \pm e_i\pm e_j \mid i,j\in \underline l\}$ with basis $\Delta:=\{2e_1, e_{ i+1}-e_i\mid 2\leq   i\leq l\}$ as subsets of $\oplus_{i\in\underline l}^\perp\R e_i$, see  \cite[1.8.8]{gorensteinlyonssolomonIII}. Recall the identification of the Weyl group $W_\Phi$ with the group ${\cal S}_{\pm\underline l}$ of permutations $\sigma$ of $\underline l\cup -\underline l$ satisfying $\sigma(-x)=-\sigma(x)$ for any $x\in\underline l\cup -\underline l$, see \cite[1.8.8]{gorensteinlyonssolomonIII}. For $\Psi$ a subset of $\Phi$ one denotes by $W_\Psi$ the subgroup of $W_\Phi$ generated by the corresponding reflections.
	
	The Chevalley generators $\bg{x}_\alpha(t)$, $\bg{n}_\alpha(t')$ and $\bg{h}_\alpha(t')$ ($\alpha\in\Phi$, $t,t'\in\overline \F_q$ with $t'\neq 0$) together with the Steinberg relations describe the group structure of $\bg{G}$, see \cite[Thm.~1.12.1]{gorensteinlyonssolomonIII}. Let $F:\bg{G}\rightarrow\bg{G}$ be the Frobenius endomorphism with $\bg{x}_{\alpha}(t)\mapsto \bg{x}_{\alpha}(t^q)$ and $G:=\bg{G}^F$, $T:={\bg T}^F$. We take for $\wt {\bg G}$ the usual conformal symplectic group ${\mathrm{CSp}}_{2l}(\overline\F_q)$. 
	
	Let $\bg{L} ={\bg T}\spann<{\bg X}_\al\mid\al\in\Phi '>$ be a standard Levi subgroup of $\bg{G}$ associated with $\Phi':=\Phi\cap\R \Delta '$ for some $\Delta '\subseteq\Delta$. Then $\Phi'$ decomposes as a disjoint union of irreducible root systems, i.e., 
	$$\Phi'=\Phi_{-1}\sqcup  \Phi_2\sqcup \ldots \sqcup\Phi_{l-1},$$
	where $\Phi_{-1}$ denotes a root subsystem with a long root, hence of type ${\sf A}_1$ or ${\sf C}_m$ ($m\geq 2$),  and $\Phi_{d}$ is the union of direct summands subsystems of $\Phi'$ of type ${\sf A}_{d-1}$ ($d\geq 2$) with only short roots. Denote $ L=\bg{L}^F$.
	
	Note that with the notation of Sect. 3.1, $E=\spann<F_0>$. Note that every standard Levi subgroup $L$ is $E$-stable. All automorphisms of $\bg{G}^F$ are induced by 
	$\wt {\bg{G}}^F\rtimes E$ as soon as $q\geq 3$. Recall that one calls diagonal the ones induced by $\wt {\bg{G}}^F$.
	
	Write $\bb{D}:=\underline l \cup\{-1\}$. For each $d\in \bb{D} \setminus \{1\}$ let $J_d\subseteq \underline l$ be minimal with $\Phi_d\subseteq \Spann<e_k| k\in J_d>$. In addition let $J_1:=\underline l \, \setminus (J_{-1}\cup J_2\cup \ldots \cup J_l)$. Then  
	$\Phi_d=\Spann<e_k|k\in J_d> \cap \Phi'$ and we denote
	$$
	\overline  \Phi_d:=\Spann<e_k|k\in J_d> \cap \Phi
	$$ for every $d\in \bb{D}$. For $d\in \bb{D}\setminus\{1\}$ let $\mathcal{O}_d$ be the set of $W_{\Phi_d}$-orbits in $J_d$, and let $\mathcal{O}_1:=\{\{j\} \mid j\in J_1\}$. Let 
	$$ \mathcal{O}:=\bigcup_{d\in \bb{D}}\mathcal{O}_d.$$
\end{notation}

The following lemmas gather facts that are easily checked by use of the Steinberg relations or the realization of $G$ as $\bt{Sp}_{2l}(q)$ given in \cite[2.7]{gorensteinlyonssolomonIII}. 

\begin{lemma}\label{lem33_C}
	For $I\subseteq \underline l$ let $\bg{T}_I:=\Spann<\bg{h}_{2e_i}(t)|i\in I,\ t\in \overline \F_q^\times>$. For each $I\in \mathcal{O}_d$ with $d\not=1$ let  
	$\Phi_I:=\Phi_d\cap \Spann<e_i|i\in I> 
	$, 
	\begin{align}\label{def_bGIC}
	\bg{G}_I:=\spann<\bg{X}_{\alpha}\mid \alpha\in\Phi_I>\bg{T}_I \,\, \text{ and } \, \, G_I=\bg{G}_I^F.
	\end{align}
	\begin{thmlist}
		\item Then $G_I\cong \bt{GL}_{|I|}(q)$ if $I\neq J_{-1}$ and  $G_{J_{-1}} \cong\bt{ Sp}_{2|J_{-1}|}(q)$.
		\item $L$ is the direct product of the groups $G_I$ ($I\in \mathcal{O}$).
		\item $\wt L$ induces 
		diagonal automorphisms on  $G_{J_{-1}}$ and only inner automorphisms on $G_I$ ($I\neq J_{-1}$).
	\end{thmlist}
\end{lemma}

\begin{lemma}\label{lem34_C}
	Let $\bg{h}_I(-1):=\prod_{j\in I}\bg{h}_{2e_j}(-1)$ for $I\subseteq \underline l$, 
	\begin{align}
	\label{def_HC}
	H:=\Spann<\bg{h}_I(-1)|I \in \mathcal{O}> \text{ and } H_d=\Spann<\bg{h}_I(-1)| I\in\mathcal{O}_d> \, \text{($d\in \bb{D}$).}
	\end{align}
	Then  $H=H_{-1}\times H_1\times H_2\times \cdots\times H_l$ and  $H\leq \Zent(L)$
\end{lemma}

We keep the same notations as before. Recall that we identify $W_{\Phi}$ with the group $\Sym_{\pm \underline l}$ defined in \cite[1.8.8]{gorensteinlyonssolomonIII}.

\begin{proposition}\label{prop64_C} We have $N/L\cong  W_{\ov\Phi_1}\times \prod_{d\geq 2} \bt{Stab}_{W_{\ov \Phi_d}}(\Phi_d)$. Moreover $$\operatorname{Stab}_{W_{\ov  \Phi_d}}(\Phi_d)= \left (W_{\Phi_d} \times \Spann<\prod_{i\in I} (i,-i)| I\in \mathcal{O}_d >\right )\rtimes \Sym_{\mathcal{O}_d} $$ for $2\leq d\leq l$. 
\end{proposition}
\begin{proof}This follows from  $N/L\cong \NNN_N(\bg T)/\NNN_L(\bg T)\cong\NNN_{W} (W_{\Phi'})/W_{\Phi'}$, see \cite[9.2.2]{Carter2}. The computation of stabilizers in root systems of type $\sf C$ is standard.
\end{proof}

\begin{notation}[Introduction of $V_d$]\label{not38_C}
	We write $\bg{n}_i:=\bg{n}_{\alpha_i}(-1)$ whenever $\alpha_1=2e_1$ and $\alpha_i=e_i-e_{i-1}$ ($2\leq i \leq l$). Note that the elements $\{ \bg{n}_i|1\leq i \leq a\}$ satisfy the braid relations of type ${\sf C}_a$, see for example \cite[9.3.2]{Springer98}. 
	
	For $d\in \underline l$, let $a_d:=|\mathcal{O}_d|$, $I_{d,j}$ ($1\leq j \leq a_d$) the elements of $\mathcal{O}_d$ and $I_{d,j}(k)\in I_{d,j}$ ($1\leq k \leq d$) the elements of $I_{d,j}$. For each $ k\in \underline d$ we fix $$\pi_k: \underline {a_d}\rightarrow J_d \text{ with } j\mapsto  I_{d,j}(k) \text{ and } m_k:=\prod_{j\in \uad} \bg{n}_{e_j -e_{\pi_k(j)}}(1)\in \bg{G}.$$
	For $j\in \uad$ we define 
	$$ n_j^{(d)}:=\prod_{k\in\underline d}  \bg{n}_j^{m_k}.$$
	Alternatively we write also $n_{I_{d,1}}$ for $n _1 ^{(d)}$ and 
	$ n_{I_{d,j-1} ,I_{d,j}}$ for $n_j^{(d)}$. 
\end{notation}

\begin{proposition}\label{prop56_C}
	For $d\in \bb{D}$ let 
	\begin{align}\label{def_HC_VC}
	V_d:=\Spann< n_j^{(d)} | j\in \uad>  \text{ and } V:=\Spann<V_d| d\in\bb{D}>.
	\end{align}
	\begin{thmlist}\itemsep0cm
		\item \label{prop56a_C}
		$n_j^{(d)}=\begin{cases} \prod_{k\in I_{d,1}} \bg{n}_{e_k}(\pm 1)& \text{if }j=1,\\
		\prod_{k\in \underline d} \bg{n}_{e_{I_{d,j-1}(k)}-e_{I_{d,j}(k)}}(\pm 1)& \text{ otherwise },
		\end{cases}$
		
		 for at least one choice of the signs above;
		\item $[E,V]=1$;
		\item $N=LV$;
		\item \label{braidonjd_C}
		the elements $\{ n_j^{(d)} \mid  j\in \uad\}$ satisfy the braid relations of type ${\sf C}_{a_d}$;
		\item \label{prop66_c_C} $[V_d, V_{d'}]= 1$ for every $d,d'\in \bb{D}$ with $d\neq d'$.
	\end{thmlist}
\end{proposition}
\begin{proof}
	The elements $\{ \bg{n}_k \mid k\in \uad \}$ satisfy the braid relations as recalled in Notation~\ref{not38_C}. By the definition together with the Steinberg relations it is straighforward computation to check that the elements $\{ n_j^{(d)}\mid j\in \uad\}$ satisfy parts (a), (b) and (d). 
	
	Denote $\rho\colon \NNN_{\bg G}({\bg T})\to W_\Phi$ the canonical surjection. For $d\in\bb{D}$ we see that $\rho(V_d)W_{\Phi_d} =\operatorname{Stab}_{W_{\overline  \Phi_d}}(\Phi_d)$ since 
	$$\rho(V_d)=\Spann<\prod_{i\in I} (i,-i)| I\in \mathcal{O}_d> \rtimes  \Sym_{\mathcal{O}_d},$$ 
	whenever $d\in\bb{D}\setminus \{-1\}$.  This implies (c) by Proposition \ref{prop64_C}.
	
	Note that $\rho(V_{d})\leq W_{ \overline  \Phi_{d}}$. Since $ \overline  \Phi_{d}\perp \overline  \Phi_{d'}$ and no non-trivial linear combination of them is a root, $[V_d, V_{d'}]= 1$ by the commutator formula. \end{proof}

For the later proof of assumption \ref{cor_toolext} of Proposition~\ref{cor_tool} we need to analyze the action of $V$ on $L$. 

\begin{lemma}[The action of $V$ on $L$]\label{VonL}
	 
	\begin{thmlist}
 		\item Let $I,I' \in \mathcal{O}\setminus\{J_{-1}\}$ such that $n_{I,I'}$ is defined and $I''\in\mathcal{O}$. Then $n_{I,I'}^2\in \Zent(L)$	and
		$$[ n_{I,I'} , G_{I''}]=
		\begin{cases}
		G_I& \text{if } I''=I',\\
		G_{I'}& \text{if } I''=I,\\
		G_{I''}& \text{ otherwise }.\\
		\end{cases}$$
		\item Let $I \in \mathcal{O}\setminus\{J_{-1}\}$ and $I''\in\mathcal{O}$ with $I''\neq I$. Then $ (G_{I''})^{n_I}=G_{I''}$. The element $n_I$ induces on $G_I$ the combination of a graph and an inner automorphism while acting trivially on $G_{I''}$ if $I\neq I''$. 
	\end{thmlist}
\end{lemma}
\begin{proof}
	The claims follow from Proposition~\ref{prop56_C}\ref{prop56a_C} using the Steinberg relations. \end{proof}

\subsection{Cuspidal characters of $L$ and their extensions}
In the following we verify the character theoretic assumptions necessary for applying Proposition~\ref{cor_tool}.
 
\begin{proposition}\label{prop4.10} There exists an $NE$-stable $\wt L$-transversal $\bb{T}$ in $\cusp(L)$ such that $(\wt L NE)_\la=\wt L_\la (NE)_\la$ for every $\la\in\bb{T}$.

\end{proposition}
\begin{proof} 	Note first that the cuspidal characters of $L$ are the products of cuspidal characters of the $G_I$'s ($I\in \mathcal{O}$). We choose first a $\wt L$-transversal in $\cusp(G_{-1})$ that is $E$-stable. Such a transversal $\bb{T}_{-1}$ exists by \cite[3.1]{cabanesspathtypeC} and Remark~\ref{starCond}. We also know by Lemma~\ref{lem33_C} that $\wt L$ acts by inner automorphisms on all other direct factors $G_I$, so the set $\bb{T}=\cusp(L\mid \bb{T}_{-1})$ is an $E$-stable $\wt L$-transversal as required.

	Recall that for $\chi_{-1}\in\bb{T}_{-1}$ we have $V_{\chi_{-1}}=V$ and $(\wt L E)_{\chi_{-1}}=\wt L_{\chi_{-1}} E_{\chi_{-1}}$, hence altogether we see $(\wt L NE)_{\chi_{-1}}= \wt L_{\chi_{-1}} (NE)_{\chi_{-1}}$. 
	
	Let $\la\in\bb{T}$. Let $L_+:=\spann<G_d\mid d\in \bb{D}\setminus\{-1\}>$ and $\chi_+\in\Irr(L_+\mid \la)$. We have seen that $\wt L$ acts by inner automorphisms on $L_+$, hence stabilizes $\chi_{+}$ and therefore $(\wt L NE)_{\chi_{+}}= \wt L (NE)_{\chi_{+}}$. Since $\la=\chi_{-1}\chi_+$ for some $\chi_{-1}\in\bb{T}_{-1}$, the required equation holds for every $\la\in \Irr(L\mid \bb{T}_{-1})\cap \cusp(L)$. 
	\end{proof}

In the next step we show the following for the groups $H\lhd V$ from Lemma~\ref{lem34_C} and Proposition~\ref{prop56_C}. 
\begin{proposition}\label{prop_maxextHV_C} Every element of $\Irr(H)$ extends to its stabilizer in $V$. In particular there exists a $VE$-equivariant extension map (see Definition~\ref{ExtMap}) \wrt $H\lhd V$ for $\Irr(H)$. 
\end{proposition}
This will imply that the groups $H$ and $V$ satisfy the assumption  \ref{cor_tool}(ii.2) of Proposition~\ref{cor_tool}.
\begin{proof} The second statement is a consequence of the first since $E$ acts trivially on $V$ by Proposition~\ref{prop56_C}(c). So we now show that every element of $\Irr(H)$ extends to its stabilizer in $V$.
	
	By \ref{prop66_c_C} of  Proposition~\ref{prop56_C} it is sufficient to prove that for every $d\in \bb{D}$ any character of $H_d$ extends to its stabilizer in $V_d$. 
	The group $H_d$ is the $a_d$-times central products of groups $\spann< \bg{h}_I(-1)>$ ($ I \in \mathcal{O}_d$). Let $c_1^{(d)}:=\overline n_1^{(d)}$ and 
	\begin{align}\label{defcjd_C}
	c_{I_{d,j}}:=c_j^{(d)}&:= (c_{j-1}^{(d)})^{\onjd}.
	\end{align}
	In addition $n_j^{(d)}$ ($2\leq j\leq a_d$) stabilizes $\{c_I\mid I \in \mathcal{O}_d\}$.
	
	Let $\la_d\in \Irr(H_d)$. Then $\la_d$ is $V_d$-conjugate to a character $  \la_d'$ with 
	$$  \la_d'(\bg{h}_{I_{d,j}}(-1))=\begin{cases}
	-1 & \text{ if } j\leq  a',\\
	\, \ \ 1 &\text{ otherwise} ,\end{cases} $$
	for some $0\leq a'\leq a_d$. We assume that $ \la_d$ is of this form. Then 
	$$ V_{d,  \la_d}= CS \text{ , where } C:= \Spann<c_I| I\in \mathcal{O}_d > \text{ and } S:=\Spann<\onjd| j \in \uad \setminus\{ a'+1\} >.$$
	
	By the Steinberg relations we see that $[c_I,c_{I'} ]=1$ for $I,I'\in\mathcal{O}_d$. Hence one can choose an extension $\wh \la_d$ of $\la_d$ to $ H_d C $ such that 
	$$ \wh\la_d(c_I)=\wh\la_d(c_{I'}) \text{ for }I,I'\in \mathcal{O}_d.$$
	This character is accordingly $S$-stable and hence $V_{d,\la_d}$-stable. 
	
	Since by \ref{braidonjd_C} the elements $\{ n_j^{(d)} \mid 2\leq j \leq a_d\}$ satisfy the braid relations and $\rho(S)$ is the direct product of two symmetric groups, we see that $$S \cap H_d= \Spann<(n_j^{(d)})^2| j \in \uad \setminus\{ a'+1\} >. $$
	Those elements lie in the kernel of $ \la_d$. Hence there exists an extension $\psi$ of $\la_d$ to $ H_d S$ such that $S \leq \ker(\psi)$. 
	According to \cite[4.1]{spath10b} the characters $\psi$ and $\wh \la_d$ define an extension $\wt \psi$ to $V_{d,\wt \la_d}$ that  is $V_{d,  \la_d }$-invariant and extends to $V_{d,\la_d}$. 
\end{proof}

\begin{proposition} \label{propeps_C} Let $\epsilon: V \rightarrow V/H$ be the canonical epimorphism. There exists an $N E$-equivariant extension map \wrt $L\lhd L \rtimes \epsilon(V)$.\end{proposition}
In its proof we need the following observation. 
\begin{lemma}\label{lem_F0invext}
	Let $\gamma$ be an automorphism of $\bt{GL}_n(q)$ commuting with the field automorphism $F_0$ of $\bt{GL}_n(q)$. Then there exists a $\spann<\gamma,F_0>$-equivariant extension map \wrt $\bt{GL}_n(q)\lhd \bt{GL}_n(q)\rtimes \spann<\gamma>$.
\end{lemma}
\begin{proof} It clearly suffices to show that any $\chi\in\Irr(\bt{GL}_n(q))$ extends to its stabilizer in $ \bt{GL}_n(q)\rtimes \spann<F_0,\gamma>$.  By \cite[4.3.1]{BOnShin}, $\chi$ has an extension $\wt\chi$ to $\bt{GL}_n(q)\rtimes \spann<F_0>_\chi$ with $0\notin \wt\chi(\spann<F_0>_\chi)$. This implies that the various extensions of $\chi$ to $\bt{GL}_n(q)\rtimes \spann<F_0>_\chi$ have distinct restrictions to $ \spann<F_0>_\chi$. Let $A:= \spann<F_0, \gamma>_\chi$. Then $A$ is abelian and fixes $\wt \chi$ by what we have said about restrictions to $ \spann<F_0>_\chi$. On the other hand $A/ \spann<F_0>_\chi$ injects into $\spann<F_0, \gamma>/ \spann<F_0>$ hence is cyclic, so that $\wt \chi$ does extend to $\bt{GL}_n(q)\rtimes A$. This completes our proof.
\end{proof}

\begin{proof}[Proof of Proposition \ref{propeps_C}] It is sufficient to prove that there exists an $\epsilon(V_d)\spann<E>$-equivariant extension map \wrt $G_d \lhd G_d \rtimes \epsilon(V_d)$ for every $d\in \bb{D}^+$. 
	For $d=1$ the group $H_1$ is abelian and $[\epsilon(V_d),E]=1$. Hence such a map exists.
	
	Let $d\in\bb{D}_{\geq 2}$. Then $G_d\cong G_I^{a_d}$ for some ($I\in \mathcal{O}_d$) and $G_d\,\epsilon(V_d)\cong (G_I\rtimes\spann<\epsilon(c_I)>)\wr\Sym_{a_d}$ for $I\in\mathcal{O}_d$. For $d\geq 2$ the automorphism of $G_I$ induced by $\epsilon(c_I)$ commutes with $E$ and there exists an $\epsilon(c_I)\spann<E>$-equivariant extension map \wrt $G_I\lhd G_I\spann<\epsilon(c_I)>$ by Lemma \ref{lem_F0invext}.
	
	From the knowledge of the representations of wreath products we know there exists an $\epsilon(V_d)\spann<E>$-equivariant extension map \wrt $G_d \lhd G_d \rtimes \epsilon(V_d)$.
\end{proof}

We can now prove the following.
\begin{proposition} \label{thm_ext_LN_c} There exists an $NE$-equivariant extension map $\Lambda$ \wrt $L\lhd N$ for $\Irr(L)$, such that $\Lambda(\la^t)=\Lambda(\la)^t$ for every $t\in \wt T$ and $\la\in\Irr(L)$ with $\la^t\neq \la$. 
\end{proposition}
\begin{proof} We check that all the assumptions of Proposition~\ref{cor_tool}  are satisfied with $K_0=K=L$, $M=N$, $V$ as defined in Proposition~\ref{prop56_C} and $\bb{T}$ from Proposition~\ref{prop4.10}. The group theoretic assumptions are clear.
	Proposition~\ref{prop_maxextHV_C} implies that the assumption  \ref{cor_tool}(ii.2) is satisfied while Proposition~\ref{propeps_C} gives  \ref{cor_tool}(iii). We obtain an extension map $\Lambda_0$ for $\bb{T}$ and then deduce an extension map for $\Irr(L)$ by setting 
  $\Lambda(\la^t):=\Lambda(\la)^t$ for every $t\in \wt T$ and $\la\in\mathbb T$ with $\la^t\neq \la$ since $\bb T$ is a $\wt T$-transversal in $\Irr(L)$. To show that $\Lambda$ is $NE$-equivariant, note first that $[\wt T,NE]\leq L\Zent(\wt G)$. This is because $[\wt T,N]\leq [\wt G,\wt G]\cap \wt T\leq T$ and $[\wt T,E]\leq T\Zent(\wt G)$ since $F_0$ acts trivially on $\wt T/T\Zent(\wt G)\leq\wt G/G\Zent(\wt G)$ the latter being of order 2. Now let $x\in NE$, $\la\in\Irr(L)$ and let us check $\Lambda (\la^x)=\Lambda (\la)^x$. We have it when $\la\in \bb T$, so let us assume $\la\in\Irr (L)\setminus \bb T$. Since $\bb T$ is a $\wt T$-transversal in $\Irr(L)$ we have $\la\neq {}^t\la\in \bb T$ for some $t\in\wt T$. Denote $\mu ={}^t\la\in \bb T$. We must prove $\Lambda (\mu ^{tx} )=\Lambda (\mu ^{t} )^x$. The right hand side equals $\Lambda (\mu )^{tx} $ since $\mu^t\neq \mu\in \bb T$. For the left hand side we have seen that $[t,x]\in L\Zent(\wt G)$ hence fixes $\mu$, so $\mu^{tx}=\mu^{xt}\neq \mu^x$ while $\mu^x\in \bb T$. So $$\Lambda (\mu ^{tx} )=\Lambda (\mu ^{xt} )=\Lambda (\mu ^{x} )^t=\Lambda (\mu )^{xt}=\Lambda (\mu )^{tx}$$ the last equality since ${[t,x]}$ acts trivially on $\Irr (N_\mu)$.
\end{proof}

In our checking of the inductive Alperin-McKay conditions via Proposition~\ref{prop:ourconds}, we now have assumption \ref{prop:ourconds}(1) for the transversal whose existence is ensured by Proposition~\ref{prop4.10}. In the following, we turn to assumption \ref{prop:ourconds}(2) which deals with the so-called reflection subgroup $R(\lambda)$ of $W(\lambda):=N_\lambda)/L$ (see \cite[10.6.3]{Carter2}). The group $R(\lambda)$ is seen as acting on $\R \Phi/\R \Phi'$ and generated by reflections $s_\al$ for $\al$ ranging over a certain root system $\Phi_\la$ of $\R \Phi/\R \Phi'$.

\begin{lemma} \label{lem4.15} Let $\la\in\cusp(L)$ and $\wt \la\in\Irr(\wt L_{\la}\mid\la)$. Then $R(\la)\leq W(\wt \la)$. \end{lemma}
\begin{proof} The group $G \wt L_\la =({\bg G} \wt L_\la )^F$ has a split $BN$-pair obtained by intersection with the one of $\wt G$ and standard Levi subgroups correspond. Then $(\wt L_{\la},\wt\la)$ is a cuspidal pair for reasons already seen in (e) of the proof of Proposition~\ref{prop:ht0cusp}. This gives the meaning of $W(\wt\la)$ as a subgroup of $\NNN_{{\bg G}\wt L_\la}({\bg L}\wt L_\la)^F/\wt L_\la =N/L$.
	
	Now to prove our claim, it suffices to check that $s_\al\in W(\wt \la)$ for every $\al\in\Phi_\la$. Recall that for any $\al\in\Phi_\la$, one defines a Levi subgroup $L_\al$ of $G$ as generated by $L$ and the $\bg X^F_\beta$'s for $\beta\in \Phi$ with $\al\in \R\beta+\R\Phi '/\R\Phi '$ (see \cite[p. 330]{Carter2}). By the definition of $\Phi_\la$ the character $R_L^{L_\al}(\la)$ has two constituents of different degrees (see \cite[Sect. 10.6]{Carter2}). Now there exists an extension $\wt \la$ of $\la$ to $\wt L_\la$ since $\wt L/L\cong \wt G/G$ is cyclic. Again by the compatibility of Harish-Chandra induction with regular embeddings and intermediate inclusions, one has
	$$ \res_{L_\al }^{L_\al\wt L_\la} \circ R^{\wt L_\la L_\al }_{\wt L_\la}(\wt \la)= R_{ L}^{L_\al}(\la).$$ 
	Because of $L_\al\lhd \wt L_\la L_\al$, $R_{\wt L_\la}^{\wt L_\la L_\al}(\la)$ must also have two constituents of different degrees by Clifford theory. This implies $s_\al\in W(\wt \la)$.
\end{proof}

We now turn to the condition of Proposition~\ref{prop:ourconds} on the linear character $\delta_{\la,\sigma}$ of $N_{^\sigma\la}$ introduced in \cite[p. 887]{MalleSpathMcKay2} and whose definition is recalled in the proof below.

\begin{proposition}\label{propRlambda}   We have $R(^\sigma\la)\leq \ker(\delta_{\la,\sigma})$ for any $\sigma\in  \Aut(G)$ induced by an element of $\wt T$. \end{proposition}

\begin{proof} Thanks to Lemma~\ref{lem4.15}, it suffices to check that $W(\wt\la)\leq \ker(\delta_{\la,\sigma})$ for some $\wt \la\in \Irr(\wt L\mid\la)$. Let us recall the extension map $\Lambda$ \wrt $L\lhd N$ for $\Irr(L)$ from Proposition~\ref{thm_ext_LN_c} so that $\delta_{\la,\sigma}$ is uniquely defined as the linear character of $N_{^\sigma\la}$ satisfying
	$$\delta_{\la,\sigma}{}\,\Lambda(^\sigma\la) ={}^\sigma(\Lambda(\la)).\eqno(7)$$

	 By Proposition~\ref{prop4.10} we know that there exists some $NE$-stable $\widetilde L$-transversal $\mathbb T$ in $\cusp(L)$ and we may assume $\la\in \mathbb T$. Accordingly $(N \wt L )_{ \la}= N_{ \la } \wt L _{\la} $ and $(N \wt L )_{ \wt \la}= N_{ \wh \la } \wt L _{\la}$ where $\wh \la\in\Irr(\wt L_\la\mid \la)$ with $\wh \la^{\wt L}=\wt \la$. Note $N_{\wt \la}=N_{\wh \la}$.
According to \cite[4.1(a)]{spath10b} there exists a unique extension $\phi$ of $\wh \la$ to $N_{\wh \la } \wt L_\la$ with $\phi|_{N_{\wh \la}}=\Lambda(\la)|_{N_{\wh \la}}$. The character $\wt \phi=\phi^{\wt N_{\wt \la}}$ is an extension of $\wt \la$. 	

Assume now that ${}^\sigma\la\neq \la$. Then by Proposition~\ref{thm_ext_LN_c} we have $\Lambda(^\sigma\la) ={}^\sigma\Lambda(\la)$ and therefore (7) implies $\delta_{\la,\si}=1$ which gives our claim. 
So we consider the case where ${}^\sigma\la =\la$. Then our claim is equivalent to the fact that $\Lambda(\la)$ and ${}^\sigma\Lambda(\la)$ have same restriction to $N_{\wt\la}$ thanks to Clifford theory (see \cite[6.17]{isaacs}). Since $\si$ stabilizes $\la$ it also stabilizes $\wt \phi$. We see that $\Lambda(\la)|_{N_{\wt \la}}$ is the unique constituent of  $\wt \phi|_{N_{\wt \la}}$  extending $\la$. 
The character $\Lambda(\la)|_{N_{\wt \la}}$ has to be $\si$-stable and this gives our claim. 
\end{proof}

\section{Proof of \prettyref{thm:iAMtypeC}}

We now finish the proof of \prettyref{thm:iAMtypeC} by an application of Proposition~\ref{prop:ourconds} in the case where $G={\rm Sp}_{2l}(q)\leq\wt G={\rm CSp}_{2l}(q)$ with $l\geq 2$ (ensuring that $G$ is the universal covering of the simple group ${\rm PSp}_{2l}(q)$), $q$ a power of an odd prime $p$ and $\ell$ a prime $\geq 5$, dividing $q-1$. Let $B$ be an $\ell$-block of $G$, which by what has been recalled before of \cite[4.1]{CabanesEnguehard99} contains the irreducible components of $R_L^G(\la)$ for $L$ a Levi subgroup of $G$ as described in Section \ref{sec:ext_C} and some $\la\in\cusp(L)\cap{\cal E}(L,\ell ')$. Then $E$ is the group generated by the automorphism of $\wt G$ consisting in raising the matrix entries to the $p$-th power. 

The existence of the $NE$-stable $\wt L$-transversal ${\Bbb T}\subseteq \cusp(L)$ is implied by Proposition~\ref{prop4.10}. Then assumption (1) of Proposition~\ref{prop:ourconds} for $\Bbb T$ is ensured by Proposition~\ref{thm_ext_LN_c}. Now
Proposition~\ref{propRlambda} gives assumption (2) of Proposition~\ref{prop:ourconds}. 

	   On the other hand, assumption (3) in Proposition~\ref{prop:ourconds} follows from \cite{TaylorSymplectics} or \cite[3.1]{cabanesspathtypeC} thanks to Remark~\ref{starCond}. Finally, assumption (4) in Proposition~\ref{prop:ourconds} also holds for $G$ since $\out(G)\cong C_2\times E$ is abelian in this case.

\section{Acknowledgements}

Part of this work was completed while the authors were in residence at the MSRI in Berkeley, California during the Spring 2018 program
on Group Representation Theory and Applications, supported by the NSF
 Grant DMS-1440140. Part was also completed at the Isaac Newton Institute for Mathematical Sciences during the Spring 2020 program Groups, Representations, and Applications: New Perspectives, supported by EPSRC grant EP/R014604/1.  The authors thank both institutes and the organizers of the programs for
making their stays possible and providing a collaborative and productive work environment.

The second-named author was also supported in part by grants from the Simons Foundation (Award \#351233) and the NSF (Award \# DMS-1801156).  She would also like to thank the Graduate School at Universit{\"a}t Wuppertal for its hospitality during her visits in August 2018 and April 2019 in the framework of the research training group
\emph{GRK 2240: Algebro-Geometric Methods in Algebra, Arithmetic and
Topology},
which is funded by the DFG. 

We thank Gunter Malle for his remarks on an early version of our manuscript. 
\bibliographystyle{alpha}
\bibliography{researchreferences}

\end{document}